\documentclass[10pt]{article}
\usepackage[english]{babel}
\usepackage{xypic,amsmath,amssymb,amsthm,amssymb}

\def\move-in{\parshape=1.75true in 5true in}

\def\PP{\mathbb P^}

\def\GG{\mathbb G}
\def\NN{\mathbb N}

\def\pn1d{E}

\def\K1d{K[x_0^{(1)},\ldots ,x_n^{(1)};\ldots ;x_0^{(d)},\ldots ,x_n^{(d)}]}
\def\PP{\mathbb P^}

\def\G{\vec{G}}

\def\min{\mathrm{min}}

\def\codim{\mathrm{codim}}
\def\expdim{\mathrm{expdim}}
\def\Sec{\mathrm{Sec}}

\def\Split{\mathrm{Split}}

\def\Sec{\mathrm{Sec}}

\def\rk{\mathrm{rk}}

\def\qed{\hfill\vbox{\hrule\hbox{\vrule\kern3pt
\vbox{\kern6pt}\kern3pt\vrule}\hrule}\bigskip}

\newtheorem{theorem}{Theorem}[section]
\newtheorem{lemma}[theorem]{Lemma}
\newtheorem{propos}[theorem]{Proposition}
\newtheorem{corol}[theorem]{Corollary}
\newtheorem{defi}[theorem]{Definition}
\newtheorem{ex}[theorem]{Example}

\newtheorem{conj}[theorem]{Conjecture}
\newtheorem{nota}[theorem]{Notation}
\newtheorem{rem}[theorem]{Remark}

\begin{document}

\title{On the variety parametrizing completely decomposable
polynomials}
\date{}
\author{Enrique Arrondo and Alessandra Bernardi}
\maketitle

\begin{quote} Abstract: The purpose of this paper is to relate
the variety parameterizing completely decomposable
homogeneous polynomials of degree $d$ in $n+1$ variables on
an algebraically closed field, called
$\Split_{d}(\PP n)$, with the Grassmannian of $n-1$
dimensional projective subspaces of $\PP {n+d-1}$. We compute
the dimension of some secant varieties to
$\Split_{d}(\PP n)$ and find a counterexample to a conjecture
that wanted its dimension related to the one of the secant
variety to $\GG (n-1, n+d-1)$. Moreover by using an invariant
embedding of the Veronse variety into the Pl\"ucker space, 
we are able to compute the intersection of $\GG (n-1,
n+d-1)$ with
$\Split_{d}(\PP n)$, some of its secant variety, the tangential
variety and the second osculating space to the Veronese
variety.
\end{quote}

\section*{Introduction}

A classic problem inspired by Waring problem in number theory
is the following: which is the least integer $s$ such that a
general homogeneous polynomial of degree $d$ in $n+1$ variables
can be written as $L_1^d+\dots+L_s^d$, where $L_1,\dots,L_s$ are
linear forms? In terms of algebraic geometry, this problem is
equivalent to find the least $s$ such that the $s$-th secant
variety of the $d$-uple Veronese embedding of $\PP n$ is the whole
ambient space. In general, it is interesting to find 
projective varieties with defective secant varieties, i.e. not
having the expected dimension. This problem has been completely
solved by J. Alexander and A. Hirschowitz (see \cite{AH}, or
\cite{BO} for a recent proof with a different approach), who found
all the defective secant varieties to Veronese varieties. Our
original problem can be rephrased in the language of tensors.
Specifically, given an $(n+1)$-dimensional vector space $W$, which
is the least integer $s$ such that a general tensor in $S^dW$ can
be written as a sum of $s$ completely decomposable symmetric
tensors? 

With this new language, it is natural to wonder about the same
problem in the case of tensors not necessarily symmetric. For
example, the case of tensors in $W_1\otimes\dots\otimes W_d$,
yields the question of studying the smallest $s$-th secant variety of
a Segre variety filling up the ambient space (see \cite{AOP1},
\cite{CGG3}, \cite{CGG4} for some known results regarding this
problem). Another interesting problem is the case in which the
tensors are skew-symmetric or, geometrically, the study of the
smallest $s$-th secant variety of a Grassmann variety filling up the
ambient Pl\"ucker space. In this case, the only known examples of
defective $s$-th secant varieties are:
the third secant varieties to  $\GG(2,6)$ --which is also isomorphic to $\GG(3,6)$-- 
and to $\GG(3,7)$
 and the fourth secant varieties to
$\GG (3,7)$ and $\GG(2,8)$ --which is also isomorphic to
$\GG(5,8)$-- (\cite{CGG1},
\cite{McG} and
\cite{AOP2}).

There is a particularly interesting numerical relation among the
different types of tensors we just mentioned. Indeed, the
dimension of the above $S^dW$ is $n+d\choose n$, which coincides
with the dimension of the space $\bigwedge^nW'$ of skew-symmetric
tensors on an $(n+d)$-dimensional vector space $W'$. Therefore the
projectivization of the space of homogeneous polyonomials of
degree $d$ in $n+1$ variables has the same dimension as the
Pl\"ucker ambient space of the Grassmannian $\GG(n-1,n+d-1)$.
Moreover, this Grassmannian has dimension $nd$, which is also
the dimension of the variety, which we will call $\Split_{d}(\PP
n)$, parametrizing those polynomials that decompose as the product
of $d$ linear forms. These coincidences led Ehrenborg to formulate
(see \cite{Eh}) the following

\begin{conj}\label{Ehr} {\rm(\textbf{Ehrenborg})}
The least positive integer $s$ such that the
$s$-th secant variety to $\GG (n-1,n+d-1)$ fills up $\PP
{{n+d\choose d}-1}$
 is the same least $s\in \NN$ such that the
$s$-th secant variety to
$\Split_d(\PP {n})$ fills up $\PP
{{n+d\choose d}-1}$.
\end{conj}

If this were true, defective secant varieties to Grassmannians
would also produce defective secant varieties to Split
varieties. It is easy to see that, if $d=2$, then the
conjecture is true (Proposition \ref{Eh true}).
Unfortunately, the other possible defective cases coming from 
Grassmannians, namely the third secant varieties to
$\Split_{4}(\PP 3)$, to $\Split_{3}(\PP 4)$ and to $\Split_{4}(\PP
4)$, and the fourth secant varieties to
$\Split_{4}(\PP 4)$, $\Split_{6}(\PP 3)$ and $\Split_{3}(\PP 6)$,
are not defective (Example \ref{contraex}). 
In particular, we get that Ehrenborg's conjecture is not true.

The starting point of this paper was to understand until which
extent Ehrenborg's conjecture remains true
and to find if there is actually a relation between the defectivity of secant varieties to  Grassmannians and the defectivity of secant varieties of Split varieties. Since M.V. Catalisano, A.V.
Geramita and A. Gimigliano conjecture in \cite{CGG1} that the only
defective secant varieties to Grassmannians are those listed
above, one could conjecture that any Split variety with $d\ne2$ has
regular secant varieties (i.e. with the expected dimension). In
fact, we were not able to find any defective case.

We thus turn to the core of Ehrenborg's conjecture and study
what is behind the numerical coincidence. Our main idea is to
identify the $(n+1)$-dimensional vector space $W$ with
$S^nV$, where $V$ is a two-dimensional vector space. Then we
use the well known isomorphism between
$\bigwedge^{d}(S^{n+d-1}V)$ and $ S^{d}(S^{n}V)$ (see
\cite{Mu}), which has a nice and classical interpretation.
Precisely, the $d$-uple Veronese variety is naturally
embedded in $\GG (n-1, n+d-1)$ as the set of
$n$-secant spaces to the rational normal curve in $\PP
{n+d-1}$. This allows to consider $\GG (n-1, n+d-1)$ and
$\Split_{d}(\PP n)$ as subvarieties of the same projective space.
Depending on the context, we will regard this space as the
Pl\"ucker space of $\GG (n-1, n+d-1)$ or the projective space
parametrizing classes of homogeneous polynomials of degree $d$ in
$n+1$ variables.

With the point of view of homogeneous polynomials, we observe
(Remark \ref{Split-osc}) that points of $\Split_{d}(\PP n)$ are
characterized by belonging to certain osculating spaces to the
Veronese variety. Hence, in order to completely understand
$\Split_{d}(\PP n)$ we will need to first understand these
osculating spaces.

The goal of this paper is to use the previous identification to
compare $\Split_{d}(\PP n)$ --or any other variety related to it,
like osculating spaces to the Veronese variety-- with
$\GG (n-1, n+d-1)$. In particular, intersecting those varieties
with $\GG (n-1, n+d-1)$, we can regard the corresponding types of
polynomials as
$(n-1)$-dimensional linear subspaces of $\PP {n+d-1}$.

We like to recall that $\Split_{d}(\PP n)$ is often called in the literature the ``\emph{Chow variety of zero cycles}'' in fact it can be also interpreted as the projection of the Segre Variety $Seg(\PP n \times \cdots \times \PP n)$ from the $GL(V )$-complement to $S^{d}V$ in $V^{\otimes d}$ to $\PP {}(S^{d}V)$ itself (see \cite{GKZ} for a wide description of Chow varieties and \cite{C} for a recent use of those variety to study the ``codimension one decomposition''). 

We start the paper with section  \ref{prelim}, in which we
introduce the preliminaries and give some first results
about $\Split_{d}(\PP n)$ without still using its relation
with $\GG (n-1, n+d-1)$. More precisely, we prove the
regularity of the secant varieties to $\Split_{d}(\PP n)$ in a
certain range not depending on $d$ (Proposition
\ref{proposizione1}). We also include in this section our
counterexample to Ehrenborg's conjecture.

In section \ref{vero e GG}, we first describe in coordinates the
embeddings of the Veronese variety, $\Split_{d}(\PP n)$ and $\GG
(n-1, n+d-1)$ in the same projective space. This allows us to give
a first general result about the intersection of $\Split_{d}(\PP
n)$ and $\GG (n-1, n+d-1)$ (Proposition \ref{un contenimento}),
which we can improve in the case $d=3$ (Proposition \ref{Xn+1
in Split}). We end with Example \ref{non in Split3P2} (which we will need later on), in which we use this geometric description to show that some particular elements of $\GG (n-1, n+d-1)$ cannot be in
$\Split_{d}(\PP n)$.

In section \ref{tg} we study the intersection between $\GG
(n-1,n+d-1)$ and the tangential variety to the $d$-uple
Veronese variety. We arrive to the precise intersection in
Corollary
\ref{tau and grass}. Since this tangential variety parameterizes
classes of homogeneous polynomials that can be written as
$L^{d-1}M$ (where $L$ and $M$ are linear forms) we can also give a
necessary condition on $M$ for $L^{d-1}M$ to represent an element
of $\GG(n-1,n+d-1)$ (Proposition \ref{grazie}). As a consequence
of the results of this section, we can compute the intersection of
$\GG (n-1,n+d-1)$ and $\Split_{d}(\PP n)$ when $d=2$.

In order to compute the above intersection when $d=3$, we will
need to study first the intersection between
$\GG (n-1,n+d-1)$ and the  second osculating space to the
Veronese variety, to which we devote  section
\ref{osculating}  (see Theorem \ref{osc 2 alla veronese} for the
precise result). With the result of this section, we eventually
give in section \ref{Split e Grass} the intersection between
$\GG (n-1,n+d-1)$ and $\Split_{d}(\PP n)$ when $d=3$ (Theorem
\ref{split3conGrass}).

We end this paper with an appendix in which we give various results
about the intersection of $\GG(n-1,n+d-1)$ with several secant
varieties to the $d$-uple Veronese variety. In particular, we
completely describe this intersection when $d=2$ and for any
secant variety. We include this appendix, even if sometimes we
just sketch the proofs, because the results we got give an idea of
how the techniques introduced in the paper can be useful.

We like to thank Silvia Abrescia for the many and useful
conversations and Maria Virginia Catalisano for suggestions
and ideas. 

During the preparation of this work, the first
author was supported by the Spanish project number
MTM2006-04785; the second author was supported by MIUR and by funds from the University of Bologna.

\section{Preliminaries and first results}\label{prelim}

Throughout all the paper, the symbol $\PP n$ will denote the
projective space over an algebraically closed field $K$ of
characteristic zero, and we will fix a system of homogeneous
coordinates $x_0,\ldots,x_n$. We also write $\GG (k,d+k)$ for the
\emph{Grassmannian of $k$-spaces in $\PP{d+k}$} and $\G(k,V)$
for the \emph{Grassmannian of $k$-spaces in $V$}. 

We will indicate for brevity the polynomial ring $K
[x_0,\ldots ,x_n]$ with $R$ and its homogeneous part of
degree $d$ with $R_d$. With this notation, $\PP{}(R_d)$ is
naturally identified with the set of hypersurfaces of degree $d$
in $\PP n$ and, in particular, $\PP{}(R_1)$ is identified
with ${(\PP n)}^*$. 

\begin{defi}{\rm
The \emph{Veronese variety} is the subset of
$\PP{}(R_d)$ parametrizing $d$-uple hyperplanes, i.e. classes of
forms that are a $d$-th power of linear forms. We will write
$\Split_d(\PP n)$ for the subset of hypersurfaces
that are the union of $d$ hyperplanes.
}\end{defi}

\begin{rem} \label{coordinate}\rm{
If we use as homogeneous coordinates for $\PP{}(R_d)$ the
coefficients of the monomials, the $d$-uple Veronese embedding 
$$\begin{array}{rcccl}\nu_d:&\PP {}
(R_1)&\hookrightarrow& \PP {}
(R_d)&=\PP {{n+d\choose d} -1}\\
&[L]&\mapsto&[L^d].&\end{array}$$ 
(whose image is the Veronese variety) can be written as
$$ (u_0: \ldots : u_n)  \mapsto  (u_0^d: u_0^{d-1}u_1:
u_0^{d-1}u_2: \ldots : u_n^d).$$

Similarly, $\Split_d(\PP n)$ is the image of the finite map
(of degree $d!$):
$$\begin{array}{rcccl}
\phi:&\PP{}(R_1)\times
\smash{\mathop{\ldots}\limits^{d}} \times
\PP{}(R_1)&\hookrightarrow&\PP{}(R_d)\\
&([L_1],\ldots ,[L_d])&\mapsto&[L_1\cdots L_d]
\end{array}$$ 
which sends the point $([u_{0}^{(1)},\ldots ,
u_{n}^{(1)}], \ldots , [u_{0}^{(d)}, \ldots , u_{n}^{(d)}
])$ to the point whose coordinates form the canonical
basis of the space $V$ of symmetric forms of
$K[u_0^{(1)},\ldots ,u_n^{(1)};\ldots ;u_0^{(d)},\ldots
,u_n^{(d)}]$ of multidegree $(1,\dots,1)$. Hence, 
$\Split_d(\PP n)$ has dimension
$nd$ and it is the image of $\PP{}(R_1)\times
\dots\times\PP{}(R_1)$ under the linear subsystem $V\subset
H^0({\cal O}_{\PP{}(R_1)\times
\dots\times\PP{}(R_1)}(1,\dots,1))$ of symmetric forms.
When $d=2$, $\Split_2(\PP n)$ can also be
regarded as the set of classes of $(n+1)\times(n+1)$
symmetric matrices of rank at most two.
}\end{rem}

\begin{defi}{\rm 
If $X\subset \PP N$ is a projective
variety of dimension $n$ then 
its  \emph{$s$-th Secant Variety} is defined as follows: 
$$\Sec_{s-1}(X):=\overline{\bigcup_{P_1,\ldots ,P_s\in X}<P_1,\ldots ,P_s>}.$$ 
Its expected dimension is 
$$\expdim (\Sec_{s-1}(X))=\min\{N, sn+s-1\}$$
but this is not always equal to $\dim  (\Sec_{s-1}(X))$
in fact there are many exceptions. When $\delta_{s-1}=\expdim
(\Sec_{s-1}(X))-
\dim(\Sec_{s-1}(X))>0$ we will say that $\Sec_{s-1}(X)$ is {\it
defective} and $\delta_{s-1}$ is called {\it defect}.
}\end{defi}

Before starting the study on the dimension of secant varieties
of Split varieties we need to introduce some important
instruments classically utilized to study secant varieties.

\begin{defi}\label{fatpoint}\rm{
If $X\subset \PP N$ is an irreducible projective variety, an
{\it $m$-fat point} (or an {\it $m$-th point}) on $X$ is the
$(m-1)$-th infinitesimal neighborhood of a smooth point $P\in
X$ and it will be denoted by $mP$ (i.e. it is the projective
scheme $mP$ defined by the ideal sheaf ${\cal
I}^{m}_{P,X}\subset {\cal O}_{X}$).}
\end{defi}

If $\dim(X)=n$ then an $m$-fat point $mP$ on $X$ is a
$0$-dimensional scheme of length ${m-1+n\choose n}$. If $Z$ is
the union of the $(m-1)$-th infinitesimal neighborhoods in
$X$ of $s$ generic smooth points on $X$, we will say for short
that $Z$ is the union of $s$ generic $m$-fat points on $X$.

The most useful (and classical) theorem for the computation
of the dimension of a secant variety of a projective variety
is the so called {\it Terracini's Lemma}.
\begin{theorem} {\rm (\textbf{Terracini's Lemma})}\label{Terracini}
Let $X$  be an irreducible variety in $\PP N$, and let
$P_1,\ldots ,P_s$  be $s$ generic points on
$X$. Then, the projectivized tangent space to $\Sec_{s-1}(X)$ at a generic point $Q\in <P_1,\ldots ,P_s>$  is the linear span in $\PP N$ of the tangent spaces $T_{P_i}(X)$ to $X$ at $P_i$, $i=1,\ldots ,s$, i.e.
$$ T_{Q}(\Sec_{s-1}(X)) = <T_{P_1}(X),\ldots ,T_{P_s}(X)>.$$
\end{theorem}

\begin{proof}
For a proof see \cite{Te} or \cite{Ad}.
\end{proof}

{}From Terracini's Lemma we immediately get a way of checking
the defectivity of secant varieties. We include the precise
result for $\Split_d(\PP n)$, although the same technique
works for arbitrary varieties with a generically finite map
to a projective space.

\begin{corol}\label{corTerracini} The secant variety
$Sec_{s-1}\left(\Split_d(\PP n)\right)$ is not defective if
and only if $s$ general $2$-fat points on $\PP{}(R_1)\times
\smash{\mathop{\ldots}\limits^{d}}\times\PP{}(R_1)$ impose
$\min\{s(dn+1),{n+d\choose d}\}$ independent conditions to
the linear system $V$ of symmetric forms of multidegree
$(1,\dots,1)$ in $K[u_0^{(1)},\ldots ,u_n^{(1)};\ldots ;u_0^{(d)},\ldots
,u_n^{(d)}]$.
\end{corol}

\begin{proof}
By Terracini's Lemma, $\dim (\Sec_{s-1}(\Split_d(\PP n)))
=\dim (<T_{P_1}(\Split_d(\PP n)),\ldots
,T_{P_s}(\Split_d(\PP n))>)$, with
$P_1,\ldots ,P_s$ general points of $\Split_d(\PP n)$.
Since the hyperplanes of $\PP {{n+d\choose d} -1}$
containing $T_{P_i}(\Split_d(\PP n))$ are those containing
the fat point $2P_i$ on $\Split_d(\PP n)$, it follows that $\dim (\Sec_{s-1}(\Split_d(\PP n)))
={n+d\choose d} -1-h^0({\cal I}_Z(1))$, where $Z$ is the
scheme union of the fat points $2P_1,\dots,2P_s$. 

On the other hand, by Remark \ref{coordinate}, $\Split_d(\PP
n)$ is the image of $\PP{}(R_1)\times
\smash{\mathop{\ldots}\limits^{d}}\times\PP{}(R_1)$ by the
finite map $\phi$ determined by $V$. Therefore $h^0({\cal
I}_Z(1))$ is the dimension of the space of forms in $V$
vanishing on $\phi^{-1}(Z)$. By the symmetry of the forms of
$V$, it is enough to take preimages $P'_1,\dots,P'_s$ of
$P_1,\dots,P_s$ by $\phi$, and $h^0({\cal
I}_Z(1))$ is still the dimension of the forms of $V$
vanishing at $2P'_1,\dots,2P'_s$. The result follows now at
once.
\end{proof}

{}From this corollary, we can prove directly the
non-defectivity of several secant varieties to $\Split_d(\PP
n)$. We start from a technical result.

\begin{lemma} \label{forme} Let
$Q_1,\dots,Q_d,P_1,\dots,P_n\in\PP{}(R_1)=\PP n$ be a
set of points in general position. Then there  exist
$dn+1$ symmetric forms $F,F_{ij}\in K[u_0^{(1)},\ldots
,u_n^{(1)};\ldots ;u_0^{(d)},\ldots ,u_n^{(d)}]$, with
$i=1,\dots,n$ and
$j=1,\dots,d$, of multidegree
$(1,\dots,1)$, such that:
\begin{enumerate}
\item[(i)] $F(Q_1,\dots,Q_d)\ne0$ while
$F(P_i,A_2,\dots,A_d)=0$ for any $i=1,\dots,n$ and any
$A_2,\dots,A_d\in\PP{}(R_1)$.
\item[(ii)] $F_{ij}(P_k,A_2,\dots,A_d)=0$ for any
$i,k=1,\dots,n$, $j=1,\dots,d$, $k\ne i$ and
$A_2,\dots,A_d\in\PP{}(R_1)$.
\item[(iii)] $F,F_{11},\dots,F_{nd}$ are independent modulo
$I^{2}$, where $I\subset K[u_0^{(1)},\ldots
,u_n^{(1)};\ldots ;u_0^{(d)},\ldots ,u_n^{(d)}]$ is the multihomogeneous ideal of
$(Q_1,\dots,Q_d)$ in $\PP{}(R_1)\times
\dots\times\PP{}(R_1)$.
\end{enumerate}
\end{lemma}

\begin{proof} For any linear form $L\in K[u_0,\dots,u_n]$,
we will denote with $\tilde L$  the symmetrized form
$$\tilde L:=L(u_0^{(1)},\ldots
,u_n^{(1)})\cdot L(u_0^{(2)},\ldots ,u_n^{(2)})\cdots
L(u_0^{(d)},\ldots ,u_n^{(d)}).$$

Since the points are in general position we can take a
linear form $L\in K[u_0,\dots,u_n]$ vanishing at
$P_1,\dots,P_n$ and not vanishing at any $Q_1,\dots,Q_d$. We
thus take $F=\tilde L$, which satisfies (i).

Similarly, for any $i=1,\dots,n$ and $j=1,\dots,d$, we can
find $L_{ij}\in K[u_0,\dots,u_n]$ vanishing at
$P_1,\dots,P_{i-1},P_{i+1},\dots,P_n,Q_j$, and we take
$F_{ij}=\tilde L_{ij}$, and clearly (ii) holds.

Finally, to prove (iii), assume that there is a linear
combination $\lambda
F+\lambda_{11}F_{11}+\dots+\lambda_{nd}F_{nd}\in I^{2}$.
Evaluating at the point $(Q_1,\dots,Q_d)$, we get
$\lambda=0$. On the other hand, taking an arbitrary
point $U\in\PP{}(R_1)$ of coordinates
$[u_0,\dots,u_n]$, and evaluating at
$(Q_1,\dots,Q_{j-1},Q_{j+1},\dots,Q_d,U)$ we get, for any
$j=1,\dots,d$, that the linear form
$$\Sigma_{i=1}^n\lambda_{ij}F_{ij}(Q_1,\dots,Q_{j-1},Q_{j+1},
\dots,Q_d,U)\in K[u_0,\dots,u_n]$$
is in the square of the ideal of $Q_i$ in $\PP{}(R_1)$. This
clearly implies that this linear form is
identically zero. Morevover, evaluating it at
each $P_i$, with $i=1,\dots,n$, we get $\lambda_{ij}=0$, which
completes the proof.
\end{proof}

\begin{propos}\label{proposizione1} If $d>2$ and $3(s-1)\leq
n$, then $\Sec_{s-1}(\Split_d(\mathbb P^n))$ is not
defective.
\end{propos}

\begin{proof} 
It is enough to apply Corollary \ref{corTerracini}. We thus
take $s$ general points $A_1,\dots,A_s\in\PP{}(R_1)\times
\smash{\mathop{\ldots}\limits^{d}}\times\PP{}(R_1)$ and need
to show that the evaluation map $\varphi:V\to H^0({\cal O}_Z)$
is surjective, where $Z$ is the subscheme of $\PP{}(R_1)\times
\smash{\mathop{\ldots}\limits^{d}}\times\PP{}(R_1)$ union of the fat points
$2A_1,\dots,2A_s$. 

For each $i=1,\dots,s$, we write $A_i=(Q_{i1},\dots,Q_{id})$.
Since $n\geq 3(s-1)$ and $d>2$, we can pick
$P_{i1},\dots,P_{in}\in\PP n$ in general position and such
that they contain the points
$Q_{j1},Q_{j2},Q_{j3}$ for any $j=1,\dots,i-1,i+1,\dots,s$. Hence the points $Q_{i1},\dots,Q_{id},P_{i1},\dots,P_{id}$ form a set of $n+d$ different points in general position to which we can apply Lemma \ref{forme}. Therefore, we can find symmetric forms
$F_i,F_{i1},\dots,F_{i,nd}\in V$ such that the image of them
under the evaluation map $\varphi$ maps surjectively to
$H^0({\cal O}_{2A_i})$. Also, the properties (i) and (ii) of
the lemma imply, together with our choice of
$P_{i1},\dots,P_{in}\in\PP n$, that these forms map to zero
in any direct summand ${\cal O}_{2A_j}$ of $H^0({\cal O}_Z)$.
Since this is true for any $i$, the surjectivity of $\varphi$
follows.
\end{proof}

We finish this section discussing Ehrenborg's conjecture.

\begin{ex}\label{contraex}
\rm{It is a known result (see for example \cite{CGG1}) that
$\Sec_{3-1}(\GG (2,6))$ has defect $\delta_2 =1$, i.e one
expects that
$\Sec_{2}(\GG(2,6))=\PP {34}$ but $\dim (\Sec_2(\GG (2,6)))=33$;
we need $\Sec_3(\GG(2,6))$ in order to fill up $\PP {34}$.
However, it is not true that the least integer $s$
such that $\Sec_{s-1}(\Split_4(\PP 3))$ fills up the ambient space
is $4$ too; in fact $\Sec_2(\Split_4(\PP 3))=\PP {34}$ (we
checked this using the previous techniques,
and making computations with \cite{CoCoA}).

In the same way, we can also prove that the third secant varieties
to $\Split_{3}(\PP 4)$ and to $\Split_{4}(\PP 4)$, and the fourth
secant varieties to
$\Split_{4}(\PP 4)$, $\Split_{6}(\PP 3)$ and $\Split_{3}(\PP 6)$,
are not defective 
}
\end{ex}

The only case for which we are able to prove that Ehrenborg's
conjecture is true is for $d=2$.
\begin{propos}\label{Eh true} The dimensions of
$\Sec_{s-1}(\GG(1,n+1))$ and $\Sec_{s-1}(\Split_2(\PP n))$ are
equal.
\end{propos}

\begin{proof}
The embedding of $\GG (1,n+1)$ into $\PP {{n+2 \choose
2}-1}{\simeq}
\PP {}(R_2)=\PP {}(K[x_{0}, \ldots , x_{n}]_{2})$ allows to look at the Grassmannian as the
set of quadrics whose representative $(n+2)\times (n+2)$ matrices
are skewsymmetric and of rank at most $2$. Therefore
$\Sec_{s-1}(\GG (1,n+1))\simeq \{ M \in M_{n+2}(K) \; | \;
M=-M^{T}, \;  \rk(M)\leq 2s \}$, then
$\codim(\Sec_{s-1}(\GG (1,n+1)))={n+2-2s\choose 2}.$

In the same way $\Split_2(\PP n)\simeq \{M\in M_{n+1}(K)\; | \; M=M^{T}, \, \rk(M)\leq 2\}$; therefore \\
$\Sec_{s-1}(\Split_2(\PP n))\simeq \{M\in M_{n+1}(K)\; | \; M
\hbox{ is symmetric and } \rk(M)\leq 2s\}$, then
$\codim(\Sec_{s-1}(\Split_2(\PP n))={n+2-2s\choose 2}=\codim(\Sec_{s-1}(\GG (1,n+1)))$.
\end{proof}

\section{Veronese varieties and Grassmannians}\label{vero e GG}

In this section we want to study the other problem inspired to
us by Ehrenborg's conjecture: the ``intersection'' between
$\GG (n-1, n+d-1)$ and $\Split_{d}(\PP n)$. To do this, we will
need to identify the ambient spaces of both varieties (see Remark \ref{identificazione}).

We collect first in a lemma the main results and definitions (written in an intrinsic way) of a classical construction that we will need in
the sequel. 

\begin{lemma} \label{nu_d(Pn)intG(n-1,n+d-1)} Consider the map
$\phi_{n,d}:\PP{}(K[t_0,t_1]_n)\to\G(d,K[t_0,t_1]_{n+d-1})$ that
sends the class of $p_0\in K[t_0,t_1]_n$ to the
$d$-dimensional subspace of $K[t_0,t_1]_{n+d-1}$ of forms of the
type $p_0q$, with $q\in K[t_0,t_1]_{d-1}$. Then the following
hold:
\item{(i)} The image of $\phi_{n,d}$, after the Pl\"ucker
embedding of $\G(d,K[t_0,t_1]_{n+d-1})$, is the $n$-dimensional
$d$-th Veronese variety. 
\item{(ii)} Identifying $\G(d,K[t_0,t_1]_{n+d-1})$ with the
Grassmann variety of subspaces of dimension $n-1$ in
$\PP{}(K[t_0,t_1]_{n+d-1}^*)$, the above Veronese variety is the
set $V$ of $n$-secant spaces to a rational normal curve
$\Sigma\subset\PP{}(K[t_0,t_1]_{n+d-1}^*)$.
\item{(iii)} For any $p\in K[t_0,t_1]_s$, with $s<n$, there is a
commutative diagram 
$$\begin{array}{ccc}
\PP{}(K[t_0,t_1]_{n-s})&\stackrel{\phi_{n-s,d}}{\longrightarrow}
&\G(d,K[t_0,t_1]_{n+d-s-1})\\
\downarrow&&\downarrow\\
\PP{}(K[t_0,t_1]_n)&\stackrel{\phi_{n,d}}{\longrightarrow}
&\G(d,K[t_0,t_1]_{n+d-1})
\end{array}$$
where the vertical arrows are inclusions naturally induced by the
multiplication by $p$.
\item{(iv)} When identifying $\G(d,K[t_0,t_1]_{n+d-1})$ with the
Grassmann variety of subspaces of dimension $n-1$ in
$\PP{}(K[t_0,t_1]_{n+d-1}^*)$, the image by $\phi_{n,d}$ of
$\PP{}(K[t_0,t_1]_{n-s})\subset\PP{}(K[t_0,t_1]_n)$ as in (iii)
is the set of $n$-secants to $\Sigma$ containing the subscheme
$Z\subset\Sigma$ defined by the zeros of $p$.
\end{lemma}

\begin{proof} Write
$p_0=u_0t_0^n+u_1t_0^{n-1}t_1+\dots+u_nt_1^n$. Then a basis of
the subspace of $K[t_0,t_1]_{n+d-1}$ of forms of the
type $p_0q$ is given by:
\begin{equation*}\left\{
\begin{array}{l}
u_0t_0^{n+d-1}+\cdots +u_nt_0^{d-1}t_1^n \\
\;\; \; u_0t_0^{n+d-2}t_1 +\cdots +u_nt_0^{d-2}t_1^{n+1} \\
 \;\; \;\; \; \ddots\\
\; \; \; \; \; \; \; \; \; u_0t_0^nt_1^{d-1}+\cdots
+u_nt_1^{n+d-1}.
\\
\end{array}
\right. \end{equation*}
The coordinates of these elements with respect to the basis
$\{t_0^{n+d-1},t_0^{n+d-2}t_1,\dots,t_1^{n+d-1}\}$ of
$K[t_0,t_1]_{n+d-1}$ are thus given by the rows of the matrix
$$\left(\begin{array}{cccccccc}
u_0&u_1&\dots&u_n&0&\dots&0&0\\
0&u_0&u_1&\dots&u_n&0&\dots&0\\
\vdots&\ddots&\ddots&\ddots&&\ddots&\ddots&\vdots\\
0&\dots&0&u_0&u_1&\dots&u_n&0\\
0&\dots&0&0&u_0&\dots&u_{n-1}&u_n
\end{array}
\right).$$
The standard Pl\"ucker coordinates of the subspace
$\phi_{n,d}([p_0])$ are the maximal minors of this matrix. It is
known (see for example \cite{AP}), these minors form a basis of
$K[u_0,\dots,u_n]_d$, so that the image of $\phi$ is indeed a
Veronese variety, which proves (i). 

To prove (ii), we still recall some standard facts from
\cite{AP}. Take homogeneous coordinates $z_0,\dots,z_{n+d-1}$
in $\PP{}(K[t_0,t_1]_{n+d-1}^*)$ corresponding to the dual basis
of $\{t_0^{n+d-1},t_0^{n+d-2}t_1,\dots,t_1^{n+d-1}\}$. Consider
$\Sigma\subset\PP{}(K[t_0,t_1]_{n+d-1}^*)$ the standard rational
normal curve with respect to these coordinates. Then, the
image of $[p_0]$ by $\phi_{n,d}$ is precisely the $n$-secant
space to $\Sigma$ spanned by the divisor on $\Sigma$ induced by
the zeros of $p_0$. This completes the proof of (ii).

Part (iii) comes directly from the definitions. Finally, in order
to prove (iv), observe that (iii) implies that the image by
$\phi_{n,d}$ of
$\PP{}(K[t_0,t_1]_{n-s})\subset\PP{}(K[t_0,t_1]_n)$ is the
subset of subspaces of $K[t_0,t_1]_{n+d-1}$ all of whose
elements are divisible by some $pp_0$ with $p_0\in
K[t_0,t_1]_{n-s}$, in particular divisible by $p$. The proof of
(ii) implies that the corresponding subspace in
$\PP{}(K[t_0,t_1]_{n+d-1}^*)$ contains the subscheme
$Z\subset\Sigma$ defined by the zeros of $p$.
\end{proof}

\begin{rem} {\rm In the above proof we used coordinates to
describe the curve $\Sigma$, because it will be useful for us
later on. However, it can be described also in an intrinsic way.
Specifically, the elements of $\PP{}(K[t_0,t_1]_{n+d-1}^*)$ are
linear forms $K[t_0,t_1]_{n+d-1}\to K$ up to multiplication by a
constant. Then $\Sigma$ is nothing but the set of classes
of linear forms of the type $F\mapsto F(a_0,a_1)$ for some
$a_0,a_1\in K$.
}\end{rem}

\begin{rem}\label{identificazione} {\rm In order to relate our
Veronese variety
$V$ with the standard Veronese variety, we will identify $R_1$
with
$K[t_0,t_1]_n$ by assigning to any $L=u_0x_0+\cdots+u_nx_n\in
R_1$ the homogeneous form
$L(t_0^n,t_0^{n-1}t_1,\dots,t_1^n)=u_0t_0^n+u_1t_0^{n-1}t_1+\cdots+u_nt_1^n\in
K[t_0,t_1]_n$. If we just write $\PP{n+d-1}$ instead of
$\PP{}(K[t_0,t_1]_{n+d-1}^*)$, the map
$\phi:\PP{}(R_1)\to\GG(n-1,n+d-1)$ sends the class of the
linear form to the subspace of $\PP{n+d-1}$ defined (in the
above coordinates) as the intersection of the hyperplanes:
\begin{equation}\label{vero in grass}\left\{
\begin{array}{l}
u_0z_0+\cdots +u_nz_n=0 \\
\;\; \; u_0z_1 +\cdots +u_nz_{n+1}=0 \\
 \;\; \;\; \; \ddots\\
\; \; \; \; \; \; \; \; \; u_0z_{d-1}+\cdots +u_nz_{n+d-1}=0 \\
\end{array}
\right. .\end{equation}
{}From now on we will use Pl\"ucker coordinates,
but in a way that is dual to the standard one. Specifically, for
any projective space $\PP{d+k}$ with homogenous coordinates
$z_0,\dots,z_{d+k}$, if $\Lambda\subset\PP{d+k}$ is the space
defined by the linearly independent equations
\begin{equation*}\begin{array}{c}
u_{1,0}z_0+\cdots+u_{1,d+k}z_{d+k}=0\\
\vdots\\
u_{d,0}z_0+\cdots+u_{d,d+k}z_{d+k}=0
\end{array}
\end{equation*}
for each $0\le i_1< \cdots < i_d\le d+k$ we define
$p_{i_1\cdots i_d}$ to be the determinant
\begin{equation*}\label{pi1id}p_{i_1\cdots i_d}:= \left|
\begin{array}{ccc}
  u_{1,i_1} & \cdots & u_{1,i_d} \\
  \vdots &  & \vdots \\
  u_{d,i_1} & \cdots & u_{d,i_d} \\
\end{array}
\right|.\end{equation*}
In this way, the Pl\"{u}cker embedding
is described as follows:
\begin{equation}\label{p}\begin{array}{rcl}
  p:\GG (k,n) & \hookrightarrow & \PP {{n+1 \choose k+1}-1} \\
  \Lambda & \mapsto & \{\{p_{i_1\cdots i_d}\}\; |\; 0\leq i_1 <
\cdots < i_d \leq d+k\}. \\
\end{array}\end{equation}
Observe that, in fact, the Pl\"ucker coordinates $p_{i_1\cdots i_d}$ obtained from \eqref{vero in grass} produce a basis of the space of the space of homogeneous polynomials of degree $d$ in the variables $u_0,u_1,\dots,u_n$. 
This yields and identification of $\PP{}(R_d)$ with the Pl\"ucker ambient space of
$\GG(n-1,n+d-1)$. When using the standard coordinates in each of
these varieties (the coefficients of the polynomial and Pl\"ucker
coordinates, respectively), this identification should be made
explicit for any concrete case, as we will show in the following example.
}\end{rem}

\begin{ex}\label{esempio n=2 d=3}{\rm
Let us make explicit the above identification 
in the case $n=2,d=3$. In this case, the map $\phi_{2,3}$ assigns to any
linear form
$u_0x_0+u_1x_1+u_2x_2$ the line of $\PP4$ given as intersection
of the hyperplanes
$$\left\{
\begin{array}{rrrrrr}
u_0z_0&+u_1z_1&+u_2z_2&&&=0 \\
&u_0z_1&+u_1z_2&+u_2z_3&&=0 \\
&&u_0z_2&+u_1z_3&+u_2z_4&=0 
\end{array}\right.$$
so that it has Pl\"ucker coordinates
$$\begin{array}{l}
p_{012}=u_0^3\\
p_{013}=u_0^2u_1\\
p_{014}=u_0^2u_2\\
p_{023}=u_0u_1^2-u_0^2u_2\\
p_{024}=u_0u_1u_2\\
p_{034}=u_0u_2^2\\
p_{123}=u_1^3-2u_0u_1u_2\\
p_{124}=u_1^2u_2-u_0u_2^2\\
p_{134}=u_1u_2^2\\
p_{234}=u_2^3.
\end{array}$$
Since the Veronese embedding $\PP{}(R_1)\to\PP{}(R_3)$ is
defined by
$u_0x_0+u_1x_1+u_2x_2\mapsto(u_0x_0+u_1x_1+u_2x_2)^3$, the
above relations show that an element of the ambient Pl\"ucker
space is naturally identified with the polynomial
\begin{equation}\label{polinomio}\begin{array}{l}
p_{012}x_0^3+3p_{013}x_0^2x_1+3p_{014}x_0^2x_2
+3(p_{023}+p_{014})x_0x_1^2+6p_{024}x_0x_1x_2+\\
+3p_{034}x_0x_2^2
+(p_{123}+2p_{024})x_1^3+3(p_{034}+p_{124})x_1^2x_2
+3p_{134}x_1x_2^2+p_{234}x_2^3.
\end{array}
\end{equation}

}\end{ex}

After the identification of Remark
\ref{identificazione}, we can restate Lemma
\ref{nu_d(Pn)intG(n-1,n+d-1)} in terms of
polynomials in $K[x_0,\dots,x_n]$.

\begin{lemma}\label{contained} 
Let $p:=a_{0}t_{0}^{s}+a_{1}t_{0}^{s-1}t_{1}+\cdots
+a_{s}t_{1}^{s}\in K [t_{0},t_{1}]_{s}$ and set, for $j=1,
\ldots , n-s+1$, the linear forms
\begin{equation*}\begin{array}{rcccccccccccc}
N_0:=&a_{0}x_0&+&a_{1}x_1&+&\cdots &+&a_{s}x_s &&&\\
N_1:=&&&a_{0}x_1&+&a_{1}x_2&+&\cdots &+&a_{s}x_{s+1} &&\\
\vdots&&&&\ddots&&&&&&\ddots\\
N_{n-s}:=&&&&&a_0x_{n-s}&+&a_{1}x_{n-s+1}&+&\cdots
&+&a_sx_n.
\end{array}\end{equation*}
Then, in the set up of Lemma
\ref{nu_d(Pn)intG(n-1,n+d-1)}, and identifying
$\PP{}(K[t_0,t_1]_n)$ with
$\PP{}(R_1)$, the inclusion
$\PP{}(K[t_0,t_1]_{n-s})\subset\PP{}(K[t_0,t_1]_n)$ is
identified with $\PP {}(K[N_0, \ldots ,
N_{n-s}]_1)\subset\PP{}(R_1)$ and its image by $\phi_{n,d}$ in
$\GG (n-1,n+d-1)$ is the
locus 
$$G':=\{ \Lambda \in \GG(n-1,n+d-1) \; | \;  \Lambda\cap \Sigma
\supseteq Z \}$$ 
where $Z\subset\Sigma$ is the subscheme defined by the
zeros of $p$. Moreover, diagram (iii) of Lemma
\ref{nu_d(Pn)intG(n-1,n+d-1)} can be written as
$$\begin{array}{ccc}
\PP{}(K[N_0,\dots,N_{n-s}]_1)&\stackrel{\phi_{n-s,d}}{\longrightarrow}
&\GG(n-s-1,n+d-s-1)\\
\downarrow&&\downarrow\\
\PP{}(K[x_0,\dots,x_n]_1)&\stackrel{\phi_{n,d}}{\longrightarrow}
&\GG(n-1,n+d-1)
\end{array}$$
where $\PP{n+d-s-1}$ is identified with the projection of
$\PP{n+d-1}$ from $<Z>$, and the natural map 
$\GG(n-s-1,n+d-s-1)\to\GG(n-1,n+d-1)$ is identified with the
inclusion of $G'$.

\end{lemma}

\begin{proof}
It is enough to recall that the subspace
$\PP{}(K[t_0,t_1]_{n-s})\subset\PP{}(K[t_0,t_1]_n)$ corresponds
to the subspace of polynomials in $K[t_0,t_1]_n$ divisible by
$p$. These polynomials take the form $(a_{0}t_{0}^{s}+a_{1}t_{0}^{s-1}t_{1}+\cdots
+a_{s}t_{1}^{s})(b_{0}t_{0}^{n-s}+b_{1}t_{0}^{n-s-1}t_{1}+\cdots
+b_{n-s}t_{1}^{n-s})$, which, as elements of $R_1$, are
precisely those of the form $b_0N_0+\dots+b_{n-s}N_{n-s}$. The
rest of the statement is obtained directly from Lemma
\ref{nu_d(Pn)intG(n-1,n+d-1)}.
\end{proof}

\begin{rem} \rm{When $s=n$, there is only one form $N_0$ and
$G'$ is just one point of $\GG(n-1,n+d-1)$, which is precisely
the point of $V$ corresponding to $[N_0^d]$.

When $s=n-1$, the set $G'$ is a projective space of dimension
$d$, so it is the whole $\PP{}(K
[N_0,N_1]_d)$. This case allows to give some first relation
between  $\Split_d(\PP n)$ and $\GG (n-1,n+d-1)$, as we do in
the following proposition.
 }\end{rem}

\begin{propos}\label{un contenimento} The intersection
$\Split_d(\PP n)\cap \GG(n-1,n+d-1)$ contains the locus of
$(n-1)$-linear spaces that are $(n-1)-secant$ to $\Sigma$.
\end{propos}

\begin{proof}
If $\Lambda$ is an $(n-1)$-secant space to $\Sigma$, then it
contains a subscheme $Z\subset\Sigma$ of length $n-1$. Hence,
Lemma \ref{contained}, implies that $\Lambda$, as an
element of $\PP{}(R_d)$, comes from a homogeneous form in
$K[N_0,N_1]_d$, so that it necessarily splits.
\end{proof}

At this point of the discussion it becomes interesting to
investigate if the previous corollary describes only an
inclusion or an equality. Let us see that, at least for $d=3$,
the intersection contains another component. We start with the
case $n=2$.

\begin{ex}\label{altra componente Split3P2}{\rm
In the set up of Example \ref{esempio n=2 d=3}, consider the
class of the polynomial $x_1(x_0-x_1)(x_1-x_2)$. This clearly
gives an element in $\PP9$ that is in $\Split_3(\PP 2)$. With
the identification given in \eqref{polinomio}, it corresponds
to the element of Pl\"ucker coordinates
$$[p_{012},p_{013},p_{014},p_{023},p_{024},p_{034},
p_{123},p_{124},p_{134},p_{234}]=
[0,0,0,2,-1,0,-4,2,0,0].$$
This point is in $\GG(1,4)$, and corresponds precisely to the
line of equations $z_0-2z_1=z_2=z_4-2z_3=0$, which does not
meet the standard rational normal curve $\Sigma\subset\PP4$.
The geometric interpretation of this line is that it is the
intersection of the following three hyperplanes:
\begin{itemize}
\item $z_2=0$, the span of the of the tangent lines of $\Sigma$
at the points $[1,0,0,0,0]$ and $[0,0,0,0,1]$,
\item $z_0-2z_1+z_2=0$, the span of the of the tangent lines of
$\Sigma$ at the points $[1,0,0,0,0]$ and $[1,1,1,1,1]$,
\item $z_2-2z_3+z_4=0$, the span of the of the tangent lines of
$\Sigma$ at the points $[0,0,0,0,1]$ and $[1,1,1,1,1]$.
\end{itemize}
We now let act the group of projectivities of $\PP{}(R_1)$ on $\Sigma$. This action is triply transitive and extends to an action as a subgroup of projectivities  of $\PP4$. As a consequence, for any
choice of different points $y_1,y_2,y_3\in\Sigma$, the
intersection of
$<T_{y_1}\Sigma,T_{y_2}\Sigma>
\cap<T_{y_1}\Sigma,T_{y_3}\Sigma>
\cap<T_{y_2}\Sigma,T_{y_3}\Sigma>$ is an element of $\GG(1,4)$
that is also in $\Split_3(\PP 2)$.

}\end{ex}

The above example can be generalized to any $n$, showing that
$\Split_3(\PP n)\cap\GG(n-1,n+2)$ contains not only the
$(n+2)$-dimensional subvariety given in Proposition \ref{un
contenimento}, but also another $(n+1)$-dimensional subvariety
(we will see in Theorem \ref{split3conGrass} that the
intersection consists exactly of those two components). We
introduce first a notation that we will use throughout the
paper.

\begin{nota} {\rm
If $\Sigma$ is a smooth curve, we will write
$\{r_1y_1,\dots,r_ky_k\}$ or $r_1y_1+\dots+r_ky_k$ to
denote the subscheme of
$\Sigma$ supported on the different points
$y_1,\dots,y_k\in\Sigma$ with respective multiplicities
$r_1,\dots,r_k$.   }\end{nota}

\begin{propos}\label{Xn+1 in Split} For any $n\ge2$, the
intersection of $\Split_3(\PP n)$ and $\GG(n-1,n+2)$ contains
the set
$$\{<Z+2y_1+2y_2>\cap<Z+2y_1+2y_3>\cap<Z+2y_2+2y_3>\ |\
Z\subset\Sigma,\  {\rm length}(Z)=n-2,\
y_1,y_2,y_3\in\Sigma\}.$$
\end{propos}

\begin{proof} Fix a subscheme $Z\subset\Sigma$ of length
$n-2$ and let $\Lambda\in\GG(n-1,n+2)$ be a subspace that can be
written as
$$\Lambda=<Z+2y_1+2y_2>\cap<Z+2y_1+2y_3>\cap<Z+2y_2+2y_3>.$$
In particular $\Lambda$ contains $Z$, so that it is contained in
the set $G'$ of Lemma \ref{contained}. Consider the projection
of $\PP{n+2}$ to $\PP4$ from $<Z>$. In this way, $\Sigma$
becomes a rational normal curve $\Sigma'\subset\PP4$, while
$\Lambda$ becomes a line $\Lambda'\subset\PP4$ that can be
written as
$$\Lambda'=<2y'_1+2y'_2>\cap<2y'_1+2y'_3>\cap<2y'_2+2y'_3>$$
where each $y'_i\in\Sigma'$ is the image of $y_i$. By Example
\ref{altra componente Split3P2}, the line $\Lambda'$ is
an element of $\Split_3(\PP 2)$. With the identifications of
Lemma \ref{contained}, this should be interpreted as follows.
The set $G'$ is identified with $\GG(1,4)$, whose Pl\"ucker
ambient space is $\PP{}(K[N_0,N_1,N_2]_3)$, so that the line
$\Lambda'$ is represented by a polynomial $F\in K[N_0,N_1,N_2]_3$
that factor into three linear forms. Hence, regarding
$\Lambda\in G'\subset\GG(n-1,n+2)$ as an element of its ambient
Pl\"ucker space $\PP{}(K[x_0,\dots,x_n]_d)$, it is represented
by the same polynomial $F\in K[x_0,\dots,x_n]_d$. Therefore
$\Lambda\in\Split_3(\PP n)$.
\end{proof}

\begin{ex}\label{non in Split3P2}{\rm
In the same way as in Proposition \ref{Xn+1 in Split}, it is
possible to prove that certain elements of $\GG(n-1,n+2)$ are
not in $\Split_3(\PP n)$. In particular, we will need later on
(see Lemma \ref{quasi}) to check that, given different points
$y_1,\dots,y_k$ on the rational normal curve
$\Sigma\subset\PP{n+2}$ and nonnegative integers
$r_1,\dots,r_k$ such that $r_1+\dots+r_k=n$, the linear
subspaces
\begin{enumerate}
\item $<(r_1+2)y_1,r_2y_2,r_3y_3\dots,r_ky_k>\cap
<r_1y_1,(r_2+2)y_2,r_3y_3,\dots,r_ky_k>\cap
<(r_1-2)y_1,(r_2+4)y_2,r_3y_3,\dots,r_ky_k>$
\item $<(r_1+2)y_1,r_2y_2,r_3y_3,\dots,r_ky_k>\cap
<r_1y_1,(r_2+1)y_2,(r_3+1)y_3,r_4y_4,\dots,r_ky_k>  \cap <(r_1-2)y_1,(r_2+3)y_2,(r_3+1)y_3,r_4y_4,\dots,r_ky_k>$
\item $<(r_1+2)y_1,r_2y_2,\dots,r_ky_k>\cap
<r_1y_1,(r_2+2)y_2,r_3y_3,\dots,r_ky_k>\cap
<(r_1-2)y_1,(r_2+3)y_2,(r_3+1)y_3,r_4y_4,\dots,r_k>$
\item $<(r_1+2)y_1,r_2y_2,\dots,r_ky_k>\cap
<r_1y_1,(r_2+2)y_2,r_3y_3,\dots,r_ky_k>\cap
<(r_1-1)y_1,(r_2-1)y_2,(r_3+4)y_3,r_4y_4,\dots,r_k>$
\item $<(r_1+2)y_1,r_2y_2,\dots,r_ky_k>\cap
<r_1y_1,(r_2+2)y_2,r_3y_3,\dots,r_ky_k>
\cap <(r_1-1)y_1,(r_2-1)y_2,(r_3+3)y_3,(r_4+1)y_4,r_5y_5,\dots,r_k>$
\end{enumerate}
have dimension $n-1$ and, as elements of $\GG(n-1,n+2)$, they
are not in $\Split_3(\PP n)$. To prove that, we first observe
that all those subspaces always contain a finite subscheme
$Z\subset\Sigma$ of length $n-2$, namely
$<(r_1-2)y_1,r_2y_2,r_3y_3\dots,r_ky_k>$ in the first three
cases and $<(r_1-1))y_1,(r_2-1)y_2,r_3y_3\dots,r_ky_k>$ in the
last two cases. Hence, projecting from $Z$, we are reduced to
the case $n=2$ and we need to check that, given points
$y_1,y_2,y_3,y_4$ in the rational normal curve in $\PP4$, the
subspaces  
\begin{enumerate} 
\item$<4y_1>\cap<2y_1,2y_2>\cap<4y_2>$
\item$<4y_1>\cap<2y_1,y_2,y_3>\cap<3y_2,y_3>$
\item$<4y_1>\cap<2y_1,2y_2>\cap<3y_2,y_3>$
\item$<3y_1,y_2>\cap<y_1,3y_2>\cap<4y_3>$
\item$<3y_1,y_2>\cap<y_1,3y_2>\cap<3y_3,y_4>$
\end{enumerate}
are lines and that, as elements of $\GG(1,4)$, they are not in
$\Split_3(\PP 2)$. By the homogeneity of $\Sigma$, we can
assume $y_1=[1,0,0,0,0]$, $y_2=[0,0,0,0,1]$, $y_3=[1,1,1,1,1]$
and $y_4=[1,\lambda,\lambda^2,\lambda^3,\lambda^4]$ with
$\lambda\ne0,1$. With this choice, the above five spaces become
respectively the lines
\begin{enumerate}
\item$z_4=z_2=z_0=0$
\item$z_4=z_2-z_3=z_0-z_1=0$
\item$z_4=z_2=z_0-z_1=0$
\item$z_3=z_1=z_0-4z_1+6z_2-4z_3+z_4=0$
\item$z_3=z_1=\lambda
z_0+(-3\lambda-1)z_1+(3\lambda+3)z_2+(-\lambda-3)z_3+z_4=0
$
\end{enumerate}
with Pl\"ucker coordinates
$[p_{012},p_{013},p_{014},p_{023},p_{024},p_{034},p_{123},p_{124},p_{134},p_{234}]$
equal to
\begin{enumerate} 
\item$[0,0,0,0,-1,0,0,0,0,0]$
\item$[0,0,0,0,-1,1,0,1,-1,0]$
\item$[0,0,0,0,-1,0,0,1,0,0]$
\item$[0,-1,0,0,0,0,6,0,-1,0]$
\item$[0,-\lambda,0,0,0,0,3\lambda+3,0,-1,0]$
\end{enumerate}
Using \eqref{polinomio}, we get respective polynomials
\begin{enumerate}
\item$-2x_1(3x_0x_2+x_1^2)$
\item$-6x_0x_1x_2+3x_0x_2^2-2x_1^3+6x_1^2x_2-3x_1x_2^2$
\item$-x_1(6x_0x_2+2x_1^2-3x_1x_2)$
\item$-3x_1(x_0^2-2x_1^2+x_2^2)$
\item$-3x_1(\lambda x_0^2-(\lambda+1)x_1^2+x_2^2)$.
\end{enumerate}
Since none of the above polynomials split into linear factors,
they do not represent points in $\Split_3(\PP 2)$.

}\end{ex}

\section{Tangential varieties to Veronese varieties and
Grassmannians}\label{tg}

We want to devote the rest of the paper to understand the
intersection of $\Split_{d}(\PP n)$ and $\GG(n-1,n+d-1)$. The
strategy will be to relate the algebraic properties of
polynomials with the geometry of subspaces in $\PP{n+d-1}$
(where we have the rational normal curve $\Sigma$ defining
$V$, thus giving the connection between the two approches).
The main idea is that a polynomial representing a point in
$\Split_{d}(\PP n)$ is characterized by having many linear
factors. This is translated in terms of geometry by means of
osculating spaces, and we will devote this section to the first
case, the tangential varities.

We recall first the background for this theory.

\begin{nota} \rm{
Denote with $O_{x}^{k}(X)$ the $k$-th osculating space to
a projective variety $X$ at the point $x\in X$, and with $\tau(X)$ the tangential variety to $X$ (observe
that $O^{0}_{x}(X)=x$ and $O^{1}_{x}(X)=T_{x}(X)$).}
\end{nota}

\begin{rem}\label{Split-osc}\rm{ We recall
from \cite{BCGI} that, for any $[L^{d}]\in V$, the
elements of $O^{k}_{[L^{d}]}(V)$ are precisely those represented
by forms of the type $L^{d-k}F$ where $F\in R_{k}$. Therefore
any point of $\Split_{d}(\PP n)$, which can be written as
$[L_{1}^{m_{1}}\cdots L_{t}^{m_{t}}]$ with $L_1,\dots,L_t\in
R_1$ different linear forms and $m_1,\dots,m_t$ 
positive integers with $\sum_{i=1}^{t}m_{i}=d$, can be obtained
as the only point in the intersection
$O_{[L_{1}^{d}]}^{d-m_{1}}(V)\cap \cdots \cap
O_{[L_{t}^{d}]}^{d-m_{t}}(V)$. Hence we have an equality
$$\Split_{d}(\PP n)=
\bigcup_{\footnotesize{\begin{array}{c}\sum_{i=1}^{t}m_{i}=d\\
\Lambda_1,\dots,\Lambda_t\in
V\end{array}}}O_{\Lambda_1}^{d-m_{1}}(V)\cap
\cdots \cap O_{\Lambda_t}^{d-m_{t}}(V)$$
where the subspaces $\Lambda_1,\dots,\Lambda_t$ are assumed to
be different. 
In the particular case $d=3$, we can simply write
$$\Split_{3}(\PP
n)=\tau(V)\bigcup\big(\bigcup_{\Lambda_{1},\Lambda_{2},
\Lambda_{3}\in V}O^{2}_{\Lambda_{1}}(V)\cap
O^{2}_{\Lambda_{2}}(V)\cap O^{2}_{\Lambda_{3}}(V)\big)$$
because any form of degree three containing a square
necessarily splits.
  }\end{rem}

In order to understand the
intersection of $\Split_{d}(\PP n)$, with $\GG(n-1,n+d-1)$, it
is therefore enough to understand the intersection of the osculating
spaces to $V$. A first geometric result in this direction is the
following.

\begin{propos}\label{punti in osculanti} Let $\Lambda$ be a
point in the osculating space $O^k_{\Lambda_0}(V)$ with $k<d$.
If we regard $\Lambda_0$ as an $n$-secant linear subspace to the
rational normal curve $\Sigma\subset\PP{n+d-1}$, then
$\Lambda_0$ contains the points (counted with multiplicity) of
the intersection $\Lambda\cap\Sigma$.
\end{propos}

\begin{proof} Let $L\in R_1$ be a linear form such that
$\Lambda_0=[L^d]$. Since $\lambda\in O^k(V)$ with $k<d$, Remark
\ref{Split-osc} implies that $\Lambda$ is represented by a form
of the type $L^{d-k}M$. 

On the other hand, let $Z\subset\Sigma$ be the schematic
intersection of $\Lambda$ and $\Sigma$ and set $s={\rm
length}(Z)$. Let $p\in K[t_0,t_1]_s$ be the polynomial whose
scheme of zeros in $\PP1$ corresponds to $Z\subset\Sigma$. By
Lemma \ref{contained}, the Pl\"ucker ambient space of the set
$G'$ of $(n-1)$-dimensional subspaces containing $Z$ is
$\PP{}(K[N_0,\dots,N_{n-s}]_d)$, for some linear forms
$N_0,\dots,N_{n-s}\in K[x_0,\dots,x_n]$. 

Hence we get 
$L^{d-k}M\in K[N_0,\dots,N_{n-s}]$. Since $d-k>0$, necessarily
$L\in K[N_0,\dots,N_{n-s}]$. Again by Lemma \ref{contained},
this implies that $\Lambda_0$ is in $G$, i.e. it contains $Z$,
as wanted.
\end{proof}

We introduce next the main tool that we will use to study the
osculating spaces to $V$ and their intersection with
$\GG(n-1,n+d-1)$.

\begin{defi}\label{Z} \rm{Consider the incidence variety
$$I:=\{(\Lambda,y)\in \GG (n-1,n+d-1)\times \Sigma\; | \;
 {\rm length}_y(\Lambda\cap \Sigma)\ge r\}\subset
\GG (n-1,n+d-1)\times \Sigma.$$
Fix $\Lambda_{0}\in \GG
(n-1,n+d-1)$ such that the intersection between
$\Lambda_{0}$ and $\Sigma$ in $\PP {n+d-1}$ is a
zero-dimensional scheme whose support at a point $y\in 
\Lambda_{0}\cap\Sigma$ has length $r$. Let $\pi_{1}$ be the
projection from $I$ to $\GG(n-1,n+d-1)$. We denote by
$Z_y\subset \PP {{n+d \choose d}-1}$ the image by $\pi_1$ of
a neighborhood of $I$ near $(\Lambda_0,y)$.}
\end{defi}

\begin{rem}\label{rami} \rm{Let
$\Lambda_0\in V$ be a point corresponding to a subspace
$\Lambda_0\subset\PP{n+d-1}$ meeting $\Sigma$ at points
$y_1,\dots,y_k$ with respective multiplicities
$r_1,\dots,r_k$ (hence $r_1+\dots+r_k=n$).  With the above
notation, each $Z_i:=Z_{y_i}$ is smooth at $\Lambda_0$ and a
neighbourhood of $V$ near $\Lambda_0$ is given by the intersection
$Z_1\cap\dots\cap Z_k$. Therefore
$$T_{\Lambda_{0}}(V)=\bigcap_{i=1}^{k}T_{\Lambda_{0}}(Z_{i}).$$
The same equality does not hold for arbitrary osculating spaces,
in which we only have one inclusion:
$$\bigcap_{i=1}^{k}O^s_{\Lambda_{0}}(Z_{i})\subset
O^s_{\Lambda_{0}}(V)$$
for any $s$. Hence, in order to study tangent or osculating spaces to
the Veronese variety $V$ we will study first those spaces for the
$Z_i$.}
\end{rem}

We devote the rest of the section to the tangent spaces to the
Grassmannian, while we will see in later sections that the
inclusion we have for second osculating spaces is enough if
$d=3$. The first step will be to compute the intersection of
$\GG(n-1,n+d-1)$ with the tangent spaces to each of the above
neighborhoods.

\begin{theorem}\label{TLxZ} Let $\Lambda\in \GG
(n-1,n+d-1)$ meeting $\Sigma$ at a zero-dimensional scheme
whose support at a point $y\in \Lambda_{0}\cap\Sigma$ has
length $r$. If $Z_y$ is as in Definition \ref{Z}, then the
intersection between the tangent space to $Z_y$ in
$\Lambda_0$ and the Grassmannian $\GG (n-1,n+d-1)$ is 
$$T_{\Lambda_0}(Z_y)\cap \GG (n-1,n+d-1)=$$ 
$$=\{\Lambda \in \GG (n-1,n+d-1) \; | \; \Lambda \supset
O^{r-1}_{x}(\Sigma), \; \dim (\Lambda 
\cap \Lambda_0)\geq n-2\}\cup$$
$$\cup\{\Lambda \in \GG (n-1,n+d-1)
\; |
\;O^{r-2}_{x}(\Sigma)\subset\Lambda\subset
<\Lambda_0,O^r_{x}(\Sigma)>
\}.$$
\end{theorem}

\begin{proof} Let the map $\PP 1 \rightarrow \PP {n+d-1}$
defined by $(t_{0},t_{1})\mapsto(t_{0}^{n+d-1},
t_{0}^{n+d-2}t_{1}, \ldots t_{1}^{n+d-1})$ be a
parameterization of $\Sigma$; without loss of generality we
may assume that $y=[1,0, \ldots , 0]\in \Sigma$ and that
$a_{1}, \ldots ,a_{r}\in K$ are such that
$\nu_{n+d-1}((t_{1}^{r}+a_{1}t_{1}^{r-1}t_{0}+ \cdots +
a_{r-1}t_{1}t_{0}^{r-1}+a_{r}t_{0}^{r})^{*})=y$. Hence
$\Lambda_{0}\in \GG (n-1,n+d-1)$ is defined in $\PP {n+d-1}$
by the equations $z_{r}=\cdots =z_{r+d-1}=0$. We will study
the affine tangent space $\hat{T}_{\Lambda_{0}}(Z_y)$ in the
affine chart of the Pl\"ucker coordinates $\{p_{r,\ldots ,
r+d-1}\neq 0\}$. Observe that in this affine chart we have a
system of coordinates given by $\{p_{r,\ldots ,\hat i,\dots,
r+d-1,j}\}$, with $i\in\{r,\ldots,r+d-1\}$ and
$j\not\in\{r,\ldots, r+d-1\}$, while the other Pl\"ucker
coordinates are homogeneous forms of degree at least two in
these coordinates.

\noindent Let $H_{i}$ for $i=1, \ldots , n+d-r$ be the
hyperplane of
$\PP {n+d-1}$ defined by the equation
$$H_i:a_{r}z_{i-1}+a_{r-1}z_{i}+ \cdots +
a_{1}z_{r+i-2}+z_{r+i-1}=0.$$ 
Hence $Z_y$ is described by
\begin{equation}\label{system H}\left\{ \begin{array}{l}
H_{1}+\mu_{1,d+1}H_{d+1}+ \cdots +\mu_{1,n+d-r}H_{n+d-r}=0\\
\vdots \\
H_{d}+\mu_{d,d+1}H_{d+1}+ \cdots +\mu_{d,n+d-r}H_{n+d-r}=0
\end{array}\right.\end{equation}
with $\mu_{i,j}\in K$ for $i=1, \ldots , d$ and $j=d+1, \ldots , n+d-r$.
\\
We want to write the matrix of the coefficients of the
previous system since it will be the matrix whose $d\times
d$ minors will give Pl\"ucker coordinates of $Z_y$. Actually
we will be interested only in $T_{\Lambda_{0}}(Z_y)$ hence we
can write such a matrix modulo all the terms of degree
bigger or equal then $2$:
\begin{equation}\label{A}A:=\left( \begin{array}{ccc}
\left.\begin{array}{ccccccc}
a_{r}&a_{r-1}&\cdots &\cdots & \cdots &a_2 & a_{1}\\
&a_{r}&a_{r-1} & && &a_{2}\\
&&\ddots &\ddots &&& \vdots \\
&&&a_{r}&a_{r-1}&\cdots&a_{d} 
\end{array}\right|
&
\left. \begin{array}{cccc}
1&0&\cdots &0 \\
a_{1}&1&&\\
\vdots &\ddots &\ddots &0\\
a_{d-1} &\cdots  &a_{1} &1
\end{array}\right|
&
\begin{array}{ccc}
\mu_{1,d+1} & \cdots & \mu_{1,n+d-r}\\
&& \\
\vdots && \vdots \\
\mu_{d,d+1}&\cdots&\mu_{d,n+d-r}
\end{array}
\end{array}
\right) .
\end{equation}
With the above system of coordinates, an affine
parametrization of
$Z_y\subset\GG(n-1,n+d-1)$ at $\Lambda_0$ is given by $p_{r,
\ldots , \hat i,
\ldots ,r+d-1,j}=\pm A_{i,j}+$quadratic terms, so that the
other Pl\"ucker coordinates are at least quadratic in the
parameters
$a_k,\mu_{l,m}$ of $Z$. Therefore an affine parameterization
of
$T_{\Lambda_{0}}(Z_y)\subset\PP{{n+d\choose d}-1}$ is
given by 
\begin{equation}\label{coordTZ}
\left\{\begin{array}{rll}p_{r, \ldots ,\hat i, \ldots
,r+d-1,j}=&\pm A_{i,j}&\cr
p_{i_1,\dots,i_d}=&0&{\rm otherwise}
\end{array}\right.\end{equation} 
with the same parameters $a_k,\mu_{l,m}$ as $Z_y$.

Therefore, the first part of \eqref{coordTZ}
shows that, if an element of $T_{\Lambda_{0}}(Z_y)$ belongs
also to $\GG(n-1,n+d-1)$, it should correspond to the linear
subspace defined by the matrix
\begin{equation}\label{B}B:=\left( \begin{array}{ccc}
\left.\begin{array}{ccccccc}
a_{r}&a_{r-1}&\cdots &\cdots & \cdots &a_2 & a_{1}\\
&a_{r}&a_{r-1} & && &a_{2}\\
&&\ddots &\ddots &&& \vdots \\
&&&a_{r}&a_{r-1}&\cdots&a_{d} 
\end{array}\right|
&
\left. \begin{array}{cccc}
1&0&\cdots &0 \\
0&1&&\\
\vdots &\ddots &\ddots &0\\
0 &\cdots  &0 &1
\end{array}\right|
&
\begin{array}{ccc}
\mu_{1,d+1} & \cdots & \mu_{1,n+d-r}\\
&& \\
\vdots && \vdots \\
\mu_{d,d+1}&\cdots&\mu_{d,n+d-r}
\end{array}
\end{array}
\right) .
\end{equation}
On the other hand, the second part of \eqref{coordTZ}
implies that the submatrix of $B$ obtained by removing the
central identity block has rank at most one. Hence
$a_r=\dots=a_2=0$, and depending on the vanishing
of $a_1$ or not, $B$ takes one of the following forms:

$$B_1=\left( \begin{array}{ccc}
0  &\cdots & 0\\
\vdots&&\vdots \\
0&\cdots&0
\end{array}\right|
\left.\begin{array}{ccc}
1& & 0\\
&\ddots& \\
0&&1
\end{array} \right|
\left. \begin{array}{ccc}
\mu_{1,d+1} & \cdots &\mu_{1,n+d-r} \\ 
\vdots &&\vdots \\
\mu_{d,d+1}&\cdots &\mu_{d,n+d-r}
\end{array}
\right) $$
with the last block of rank at most one, or

$$B_2=\left(  \begin{array}{cccc}
0 & \cdots &0&a_1 \\
\vdots &&\vdots& \\
0&\cdots &0&0
\end{array}\right|
\left.\begin{array}{ccc}
1& & 0\\
&\ddots& \\
0&&1\end{array} \right|
\left. \begin{array}{ccc}
\mu_{1,d+1} & \cdots &\mu_{1,n+d-r} \\ 
0 &\cdots&0 \\
&\cdots&\\
0&\cdots &0
\end{array}
\right).$$
Now observe that, reciprocally, the matrices of the type
$B_1$ and $B_2$ represent linear subspaces satisfying
the equations \eqref{coordTZ}, so that they are in 
$T_{\Lambda_{0}}(V)$. On the other hand, matrices of type
$B_1$ correspond to linear subspaces
$\Lambda\in
\GG (n-1,n+d-1)$ such that
$\Lambda
\supset O^{r-1}_{x}(\Sigma)$ and $\dim (\Lambda \cap
\Lambda_0)\geq n-2$, while matrices of type $B_2$ correspond
to linear subspaces
$\Lambda\in
\GG (n-1,n+d-1)$ such that
$O^{r-2}_{x}(\Sigma)\subset\Lambda\subset
<\Lambda_0,O^r_{x}(\Sigma)>$.
\end{proof}

With this result in mind, we can now compute the intersection of
$\GG(n-1,n+d-1)$ with the tangential variety to $V$. In the
statement, we will use the following notation, which we will
often repeat along the paper.

\begin{nota} {\rm
Given linear subspaces $A\subset B\in\PP{n+d-1}$ of respective
dimensions $n-2,n$, we will write
$F(A,B)$ to denote the pencil of subspaces
$\Lambda\in\GG(n-1,n+d-1)$ such that $A\subset\Lambda\subset B$.
}\end{nota}

\begin{theorem}\label{spazio tg alla veronese}
Let $\Lambda_{0}\in \GG (n-1,n+d-1)$ such that the intersection
between $\Lambda_{0}$ and $\Sigma$ in $\PP {n-1}$ is a
zero-dimensional scheme with support on $\{y_{1}, \ldots ,
y_{k}\}\subset\Sigma$ and degree $n$ such that each point
$y_{i}$ has multiplicity $r_{i}$ and $\sum_{i=1}^{k}r_{i}=n$
(obviously $1\leq k
\leq n$). Then

\begin{equation}\label{sp tg Veronese} T_{\Lambda_{0}}(V)\cap
\GG (n-1,n+d-1)=\bigcup_{i=1}^{k}
F(<O^{r_{1}-1}_{y_{1}}(\Sigma) ,\ldots
,O^{r_{i}-2}_{y_{i}}(\Sigma) ,\ldots ,
O^{r_{k}-1}_{y_{k}}(\Sigma)>, <O^{r_{i}}_{y_{i}}(\Sigma)
,\Lambda_{0} >).\end{equation}
\end{theorem}

\begin{proof}
With the notation of Remark \ref{rami}, Theorem \ref{TLxZ} shows
that, for each $i=1, \ldots , k$:
$$T_{\Lambda_{0}}(Z_{i})\cap \GG (n-1,n+d-1)=$$ 
$$=\{\Lambda \in \GG (n-1,n+d-1) \; | \; \Lambda \supset
O^{r-1}_{y_{i}}(\Sigma), \; \dim (\Lambda \cap \Lambda_{0})\geq
n-2\}\cup$$ $$\cup \{ \Lambda \in \GG (n-1,n+d-1) \; | \;
O^{r-2}_{y_{i}}(\Sigma) \subset \Lambda \subset
<\Lambda_{0},O^{r}_{y_{i}}(\Sigma)>\}.$$ Let us call for
brevity ${\mathcal A}_{i}:=\{\Lambda \in \GG (n-1,n+d-1) \; | \; \Lambda
\supset O^{r-1}_{y_{i}}(\Sigma), \; \dim (\Lambda \cap
\Lambda_{0})\geq n-2\}$ and ${\mathcal B}_{i}:= \{ \Lambda \in \GG
(n-1,n+d-1) \; | \; O^{r-2}_{y_{i}}(\Sigma) \subset \Lambda
\subset <\Lambda_{0},O^{r}_{y_{i}}(\Sigma)>\}$. By Remark
\ref{rami} we have that
$$T_{\Lambda_{0}}(V)\cap \GG (n-1,n+d-1)=\left(\bigcap_{i=1}^{k}T_{\Lambda_{0}}(Z_{i})\right)\cap \GG (n-1,n+d-1).$$
Then 
$$T_{\Lambda_{0}}(V)\cap \GG (n-1,n+d-1)=\bigcap_{i=1}^{k}{\mathcal A}_{i}\cup {\mathcal B}_{i}$$
Now it is sufficient to observe that all these intersections are equal to $\Lambda_{0}$ except for ${\mathcal A}_{1}\cap \cdots \cap \hat{{\mathcal A}}_{i}\cap \cdots \cap {\mathcal A_{k}}\cap {\mathcal B}_{i}$, for all $i=1, \ldots , k$,
that is $\{\Lambda\in  \GG (n-1,n+d-1) \; | \;
<O^{r_{1}-1}_{y_{1}}(\Sigma) ,\ldots
,O^{r_{i-1}-1}_{y_{i-1}}(\Sigma),O^{r_{i}-2}_{y_{i}}(\Sigma)
,O^{r_{i+1}-1}_{y_{i+1}}(\Sigma), \ldots ,
O^{r_{k}-1}_{y_{k}}(\Sigma)> \, \subset \Lambda \subset \,
<O^{r_{i}}_{y_{i}}(\Sigma) ,\Lambda_{0} >\}$ from which we have
the statement.
\end{proof}

\begin{rem} \rm{Observe that if  $\sharp\{y_{1}, \ldots
,y_{k}\}=\deg(\Lambda_{0}\cap \Sigma)=n$ then (\ref{sp tg
Veronese}) becomes:
$$T_{\Lambda_{0}}(V)\cap \GG
(n-1,n+d-1)=\bigcup_{i=1}^{n}F(<y_{1}, \ldots , \hat{y}_{i},
\ldots ,y_{n}>, <y_{1}, \ldots , l_{i}, \ldots , y_{n}>)$$
where $l_{i}=T_{y_{i}}(\Sigma)$.
\\
On the other hand, if ${\rm length}(\Lambda_{0}\cap \Sigma)=n$
and $y_{1}=\cdots =y_{k}$ then (\ref{sp tg Veronese}) becomes:
$$T_{\Lambda_{0}}(V)\cap \GG
(n-1,n+d-1)=F(O^{n-2}_{y_{1}}(\Sigma),
O^{n}_{y_{1}}(\Sigma)).$$ }
\end{rem}
\begin{defi} \rm{Let $X\subset \PP N$ be a projective, reduced and irreducible variety. 
Let $X_{0}\subset X$ be the dense subset of regular points of $X$. We define the tangential variety to $X$ as
$$\tau(X):=\overline{\bigcup_{P\in X_{0}}T_{P}(X)}.$$}
\end{defi}

\begin{corol}\label{tau and grass} The intersection between
tangential variety to Veronese variety $V=\nu_{d}(\PP n)$ and
the Grassmannian $\GG (n-1, n+d-1)$ is
$$\tau(V)\cap \GG (n-1,n+d-1)=$$
\begin{equation}\label{tangential and
grass}=\bigcup_{\Lambda=<r_1y_1,\dots,r_ky_k>\in V}\left(
\bigcup_{i=1}^k
F(<O^{r_{1}-1}_{y_{1}}(\Sigma) ,\ldots
,O^{r_{i}-2}_{y_{i}}(\Sigma) ,\ldots ,
O^{r_{k}-1}_{y_{k}}(\Sigma)>, <O^{r_{i}}_{y_{i}}(\Sigma)
,\Lambda>)\right).\end{equation} 
\qed 
\end{corol}

Observe that, when $d=2$, we have $\tau(V)=\Sec_1(V)=\Split_2(\PP
n)$, so that the above corollary also gives the intersection of
$\GG (n-1,n+1)$ with $\Sec_1(V)$ and $\Split_2(\PP n)$.

Since elements of the tangent space to $V$ at $[L^d]$ take the
form $[L^{d-1}M]$, one can wonder whether it is possible to
give some information about the linear form $M$. We conclude
this section answering that question.

\begin{propos}\label{grazie} Let $[L_0^d]\in V$ be an element
corresponding to an $n$-secant subspace
$\Lambda_0\subset\PP{n+d-1}$ to
$\Sigma$. Then, if $\Lambda\in T_{\Lambda_0}(V)\cap\GG(n-1,n+d-1)$
is given by
$[L_0^{d-1}L_1]$, the point $[L_1^d]\in V$ corresponds to a
linear space
$\Lambda_1\subset\PP{n+d-1}$ sharing with $\Lambda_0$ a
subscheme of $\Sigma$ of length $n-1$.
\end{propos}

\begin{proof} By Theorem \ref{spazio tg alla veronese}, we have
that $\Lambda$ shares with $\Lambda_0$ a subscheme
$Z\subset\Sigma$ of length $n-1$. On the other hand, the fact
that $\Lambda$ corresponds to $L_0^{d-1}L_1$ implies (see Remark
\ref{Split-osc}) that $\Lambda\in O_{[L_1^d]}^{d-1}(V)$. Hence,
by Proposition \ref{punti in osculanti}, it follows that
$\Lambda_1$ contains $Z$, as wanted.
\end{proof}

\section{Second osculating space to the Veronese Variety and the
Grassmannian}\label{osculating}

We devote this section to study the intersection with the
Grassmanniann of the second osculating space to the Veronese
variety. As we have seen, in the case of the first osculating
space (i.e. the tangential variety), the computations were
difficult to manage. In fact, the case of the second osculating
space is maybe the last handy case with these techniques,
although only the case $d=3$ seems to be treatable.

\begin{theorem}\label{O2Zramo} Let $\Lambda_0\in \GG (n-1,n+2)$
such that the intersection $\Lambda_0\cap \Sigma\subset \PP {n+2}$ is
a zero-dimensional scheme whose support contains $x\in \Sigma$ with
multiplicity $r$. Let $Z_y$ be as in Definition \ref{Z} with
$d=3$. Then the intersection between the second osculating
space to $Z_y$ in
$\Lambda_0$ and the Grassmannian $\GG (n-1,n+2)$ satisfies 
$$O^{2}_{\Lambda_0}(Z_y)\cap \GG (n-1,n+2)\subseteq {\mathcal
A}\cup {\mathcal B} \cup {\mathcal C}$$ where 
$${\mathcal A}=\{\Lambda \in \GG (n-1,n+2) \; | \; \Lambda\subseteq
<\Lambda_{0}, O^{r+1}_{x}(\Sigma)>, \, \dim (\Lambda \cap
O^{r-1}_{x}(\Sigma))\geq r-2\}$$
$${\mathcal B}=\{\Lambda \in \GG (n-1,n+2) \; |
\;O^{r-2}_{x}(\Sigma)\subseteq\Lambda, \dim (\Lambda \cap
O^{r}_{x}(\Sigma))\geq r-1, \, \dim (\Lambda \cap
<\Lambda_{0},O^{r}_{x}(\Sigma)>)\geq n-2\}$$
$${\mathcal C}= \{\Lambda \in \GG (n-1,n+2) \; | \;
O^{r-1}_{x}(\Sigma)\subseteq\Lambda , \, \dim (\Lambda \cap
\Lambda_{0})\geq n-3\}.$$
\end{theorem}

\begin{proof} As in Theorem \ref{TLxZ} we give a parameterization of
$\Sigma$ around the point $x:=[1,0,\ldots ,0]$, and we give the
descripition of $Z_y$ via the system (\ref{system H}), that, in
this case for $d=3$, becomes
$$\left\{ \begin{array}{l}
H_{1}+\mu_{1,4}H_{4}+ \cdots +\mu_{1,n+3-r}H_{n+3-r}=0\\
H_{2}+\mu_{2,4}H_{4}+ \cdots +\mu_{2,n+3-r}H_{n+3-r}=0\\
H_{3}+\mu_{3,4}H_{4}+ \cdots +\mu_{3,n+3-r}H_{n+3-r}=0
\end{array}\right.$$

Next we have to consider the matrix $A$ defined
in \eqref{A}, but now we have to keep the terms of degree two.
Depending on whether $r\ge3$ or $r=1,2$ the form of the matrix is
different, so that we will distinguish the three cases.

{\bf CASE $r\ge 3$:} In this case the matrix $A$ takes the form:

\begin{equation}\label{A again}A=\left(\begin{array}{ccc}
a_{r}&a_{r-1} &a_{r-2}\\
0 & a_{r} & a_{r-1}\\
0&0&a_{r}\end{array}\right.
\left|\begin{array}{ccc}
a_{r-3}+\mu_{1,4}a_{r} &\cdots&a_{1}+\sum_{i=4}^{r}\mu_{1,i}a_{i}\\
a_{r-2}+\mu_{2,4}a_{r} &\cdots&a_{2}+\sum_{i=4}^{r}\mu_{2,i}a_{i}\\
a_{r-1}+\mu_{3,4}a_{r} &\cdots&a_{3}+\sum_{i=4}^{r}\mu_{3,i}a_{i}\end{array}\right| \end{equation}
$$
\left|\begin{array}{ccc}
1+\sum_{i=3}^{r}\mu_{1,i+1}a_{i}&0+\sum_{i=2}^r\mu_{1,i+2}a_i&0+\sum_{i=1}^{r}\mu_{1,i+3}a_{i}\\
a_{1}+\sum_{i=3}^{r}\mu_{2,i+1}a_{i}&1+\sum_{i=2}^r\mu_{2,i+2}a_i&0+\sum_{i=1}^{r}\mu_{2,i+3}a_{i}\\
a_{2}+\sum_{i=3}^{r}\mu_{3,i+1}a_{i}&a_{1}+\sum_{i=2}^r\mu_{3,i+2}a_i&1+\sum_{i=1}^{r}\mu_{3,i+3}a_{i}
\end{array}\right.
\left|\begin{array}{ccc}
\mu_{1,4}+\sum_{i=1}^{n-r-1}\mu_{1,i+4}a_{i}&\cdots&\mu_{1,n+3-r}\\
\mu_{2,4}+\sum_{i=1}^{n-r-1}\mu_{2,i+4}a_{i}&\cdots&\mu_{2,n+3-r}\\
\mu_{3,4}+\sum_{i=1}^{n-r-1}\mu_{3,i+4}a_{i}&\cdots&\mu_{3,n+3-r}
\end{array}\right).$$
(We apologize with the reader but the matrix $A$ is too big
to be written on only one line: it is a $(3\times (n+3-r))$
size and we write first the firsts $r$ columns and secondly
the others.)

{}From this matrix, and proceeding as in the proof of Theorem
\ref{TLxZ}, one could get an affine parametrization of
$O^{2}_{\Lambda_{x}}(V)\subset\PP{{n+3\choose3}-1}$ in the
affine open set $p_{r,r+1,r+2}\ne0$. However, such a
parameterization becomes too complicated, so that we just write
the part that we need to get the result:

\begin{itemize} 
\item $p_{j,r+1,r+2}=\left\{ \begin{array}{ll}
a_{r-j}, & j=0,1,2 \\
a_{r-j}+\overline{a_{r-j+3}\mu_{1,4}}+\cdots
+\overline{a_r\mu_{1,j+1}}, & j=3,\dots , r-1 \\
\mu_{1,j-r+1}+\overline{a_1\mu_{1,j-r+2}}+\cdots
+\overline{a_{n-j+2}\mu_{1,n-r+3}}, & j=r+3, \dots , n+2, 
 \end{array} \right.$
\item $-p_{j,r,r+2}=\left\{\begin{array}{ll}  
-\overline{a_ra_1}, & j=0 \\
-\overline{a_{r-j}a_1}+a_{r-j+1}, & j=1,2 \\
-\overline{a_{r-j}a_1}+a_{r-j+1}+\overline{a_{r-j+3}\mu_{2,4}}+\cdots
+\overline{a_r\mu_{2,j+1}}, & j=3, \dots , r-1 \\
-\overline{a_1\mu_{1,j-r+1}}+\mu_{2,j-r+1}+\overline{a_1\mu_{2,j-r+2}}+\cdots
+\overline{a_{n-j+2}\mu_{2,n-r+3}}, & j=r+3, \dots , n+2,
\end{array}\right.$
\item $p_{j,r,r+1}=\left\{\begin{array}{ll}
-\overline{a_ra_2}, & j=0 \\
-\overline{a_{r-1}a_2}-\overline{a_ra_1}, & j=1 \\
-\overline{a_{r-2}a_2}-\overline{a_{r-1}a_1}+a_r, & j=2 \\
-\overline{a_{r-j}a_2}-\overline{a_{r-j+1}a_1}+a_{r-j+2}+\overline{a_{r-j+3}\mu_{3,4}}+\cdots
+\overline{a_r\mu_{3,j+1}}, & j=3, \dots , r-1 \\
-\overline{a_2\mu_{1,j-r+1}}-\overline{a_1\mu_{2,j-r+1}}+\mu_{3,j-r+1}+
\overline{a_1\mu_{3,j-r+2}}+\cdots +
\overline{a_{n-j+2}\mu_{3,n-r+3}}, & j=r+3,
\dots , n+2.
\end{array}\right.$
\item $p_{0,r-1,r+2}=\overline{a_{r}a_{2}}=-p_{0,r,r+1}$;
\item $p_{1,r-1,r+2}=\overline{a_{r-1}a_{2}}-\overline{a_{r}a_{1}}=-p_{1,r,r+1}-2p_{0,r,r+2}$;
\item $p_{2,r-1,r+2}=\overline{a_{r-2}a_{2}}-\overline{a_{r-1}a_{1}}=-p_{2,r,r+1}-2p_{1,r,r+2}-p_{0,r+1,r+2}$;
\item $p_{0,1,r+1}=0$;
\item $p_{0,i,r+1}=-\overline{a_{r}a_{r-i+2}}=-p_{0,i-1,r+2}$, for $i=3, \ldots , r-1$;
\item $p_{1,2,r+1}=-\overline{a_{r}a_{r-1}}=-p_{0,2,r+2}$;
\item $p_{1,i,r+1}=-\overline{a_{r-1}a_{r-i+2}}=-p_{1,i-1,r+2}+p_{0,i+1,r+1}=-p_{1,i-1,r+2}-p_{0,i-1,r+2}$, for $i=3, \ldots , r-1$;
\item $p_{2,3,r+1}=-\overline{a_{r-2}a_{r-1}}+ \overline{a_{r}a_{r-3}}=-p_{1,3,r+2}$;
\item $p_{2,i,r+1}=-\overline{a_{r-2}a_{r-i+2}}+ \overline{a_{r}a_{r-i}}=-p_{2,i-1,r+2}+p_{1,i+1,r+1}-p_{0,i+2,r+1}$, for $i=4, \ldots, r-1$;
\item $p_{0,i,r}=0$ for all $i=1, \ldots , n+3-r$;
\item $p_{1,i,r}=p_{0,i,r+1}=p_{0,i-1,r+2}$ for $i=2, \ldots , r-1$;
\item $p_{2,i,r}=p_{1,i-1,r+2}$ for $i=3,\ldots , r+2$;
\item $p_{i,j,r+2}=p_{i+1,j+1,r}$ for $i=0, \ldots ,r-3$ and $j=1, \ldots ,r-2$;
\item $p_{r-i,r-2,r+2}-p_{r-i,r-1,r+1}=\overline{a_{i+2}a_{1}}-\overline{a_{i+1}a_{2}}=-p_{r-i-1,r-1,r+2}$, for $i=3, \ldots , r-1$.
\end{itemize}
where the bars denote new parameters corresponding to terms
of degree two in the parametrization of $Z_y$.

Now it is needless to say that applying all these relations
at the matrix $B$ (defined as in the proof of Theorem \ref{TLxZ}) is a
complete mess... At the end of the game we succeed with a matrix that
can be only of one of the following forms:
$$B'=\left( \begin{array}{ccc}
0&0&0\\
0&0&0\\
0&0&0 \end{array}\right.
\left| \begin{array}{ccccc} 
0&\cdots&0&*&*\\
0&\cdots&0&0&0\\
0&\cdots&0&0&0\end{array}\right.
\left| \begin{array}{ccc}
1&0&0\\
0&1&0\\
0&0&1\end{array}\right.
\left|\begin{array}{ccc}
*&\cdots &*\\
*&\cdots &*\\
*&\cdots &*
\end{array}\right)$$
with the condition that the rank of the submatrix obtained omitting the third block is $2$;\\
or
$$B''=\left( \begin{array}{ccc}
0&*&*\\
0&*&*\\
0&0&0 \end{array}\right.
\left| \begin{array}{ccc}
*&\cdots&*\\
*&\cdots&*\\
0&\cdots&0\end{array}\right.
\left| \begin{array}{ccc}
1&0&0\\
0&1&0\\
0&0&1\end{array}\right.
\left|\begin{array}{ccc}
*&\cdots &*\\
*&\cdots &*\\
*&\cdots &*
\end{array}\right)$$
with the conditions that at least one of the element of the second row in the firsts two blocks is different from zero, the submatrix maid by the first two blocks has rank $1$ and that that one obtained by omitting the third block has rank $2$.
\begin{description}
\item[$B'$ case:] Observe that $p_{r-2,r+1,r+2+i}=\overline{a_{2}}\cdot B_{3,r+2+i}$ for $i=1, \ldots ,n+1$. From the parameterization we get:
\begin{enumerate}
\item $p_{r-2,r+1,r+2+i}=\overline{a_{2}\mu_{3,r+2+i}}-\overline{a_{4}\mu_{1,r+2+i}}$, for $i=1, \ldots , n+1$; 
\item $p_{r-1,r,r+2+i}=\overline{a_{2}\mu_{3,r+2+i}}-\overline{a_{3}\mu_{2,r+2+i}}$ for $i=1, \ldots , n+1$; since it is equal to that we know from the description of $B'$ that is zero;
\item $p_{r-3,r+2,r+2+i}=\overline{a_{3}\mu_{2,r+2+i}}-\overline{a_{4}\mu_{1,r+2+i}}$ for $i=1, \ldots , n+1$ that we know from the description of $B'$ that is zero;
\end{enumerate}
hence $p_{r-2,r+1,r+2+i}=0$ for $i=1, \ldots , n+1$. Therefore or $\overline{a_{2}}=0$ or $B_{3,r+2+i}=0$ for all $i=1, \ldots , n+1$. Then we get the following three subcases:
$$B'_{I}:=\left( \begin{array}{ccc}
0&0&0\\
0&0&0\\
0&0&0 \end{array}\right.
\left| \begin{array}{ccccc}
0&\cdots&0&a_{2}&a_{1}\\
0&\cdots&0&0&0\\
0&\cdots&0&0&0\end{array}\right.
\left| \begin{array}{ccc}
1&0&0\\
0&1&0\\
0&0&1\end{array}\right.
\left|\begin{array}{ccc}
*&\cdots &*\\
*&\cdots &*\\
0&\cdots &0
\end{array}\right),$$
$$B'_{II}:=\left( \begin{array}{ccc}
0&0&0\\
0&0&0\\
0&0&0 \end{array}\right.
\left| \begin{array}{ccccc}
0&\cdots&0&0&a_{1}\\
0&\cdots&0&0&0\\
0&\cdots&0&0&0\end{array}\right.
\left| \begin{array}{ccc}
1&0&0\\
0&1&0\\
0&0&1\end{array}\right.
\left|\begin{array}{ccc}
*&\cdots &*\\
*&\cdots &*\\
*&\cdots &*
\end{array}\right)$$
with the condition that the rank of the submatrix obtained considering only the last two rows of the last block is $1$;
\\
and
$$B'_{III}:=\left( \begin{array}{ccc}
0&0&0\\
0&0&0\\
0&0&0 \end{array}\right.
\left| \begin{array}{ccccc}
0&\cdots&0&0&0\\
0&\cdots&0&0&0\\
0&\cdots&0&0&0\end{array}\right.
\left| \begin{array}{ccc}
1&0&0\\
0&1&0\\
0&0&1\end{array}\right.
\left|\begin{array}{ccc}
*&\cdots &*\\
*&\cdots &*\\
*&\cdots &*
\end{array}\right)$$
with the condition that the last block has rank $2$.

\item[$B''$ case:] Let $i$ be the least index such that $B''_{2,i}$ is different from zero, $i=2, \ldots , r-1$. Observe that $p_{i,r,r+2+j}=B''_{2,i} \cdot B''_{3,r+2+j}$ for all $j=1, \ldots , n+1$. As previously we get from the parameterization that 
\begin{enumerate}
\item $p_{i,r,r+2+j}=\overline{a_{r-i+1}\mu_{3,r+2+j}}-\overline{a_{r-i+2}\mu_{2,r+2+j}}$ for $j=1, \ldots , n+1$;
\item $p_{i-1,r+1,r+2+j}=\overline{a_{r-i+1}\mu_{3,r+2+j}}-\overline{a_{r-i+3}\mu_{1,r+2+j}}$ that we know from the form of $B''$ that is zero for all $j=1, \ldots , n+1$;
\item $p_{i-2,r+2,r+2+j}=\overline{a_{r-i+2}\mu_{2,r+2+j}}-\overline{a_{r-i+3}\mu_{1,r+2+j}}$  that again we know from the form of $B''$ that is zero for all $j=1, \ldots , n+1$.
\end{enumerate}
Hence $p_{i,r,r+2+j}=B''_{2,i} \cdot B''_{3,r+2+j}=0$ and, since  $B''_{2,i}$ is different from zero, we get that $B''_{3,r+2+j}=0$ for all $j=1, \ldots , n+1$. Therefore $B''$ becomes:
$$B''=\left( \begin{array}{ccc}
0&*&*\\
0&*&*\\
0&0&0 \end{array}\right.
\left| \begin{array}{ccc}
*&\cdots&*\\
*&\cdots&*\\
0&\cdots&0\end{array}\right.
\left| \begin{array}{ccc}
1&0&0\\
0&1&0\\
0&0&1\end{array}\right.
\left|\begin{array}{ccc}
*&\cdots &*\\
*&\cdots &*\\
0&\cdots &0
\end{array}\right)$$
with the condition that the first two blocks have rank $1$.
\end{description}
It is not difficult to see that the case $B'_{I}$ is contained in the
case $B''$, hence the only remaining meaningful cases are $B'_{II}$,
$B'_{III}$ and $B''$ that describe respectively the sets
${\mathcal C}$, ${\mathcal B}$, and $\{\Lambda\in \GG (n-1,n+2)\;
| \; x\in
\Lambda \subseteq <\Lambda_{0}, O^{r+1}_{x}(\Sigma)>, \, \dim
(\Lambda\cap O^{r-1}_{x}(\Sigma))\geq r-2 \}$, which is clearly
contained in ${\mathcal A}$.

\bigskip

{\bf CASE $r=2$:} The analogous of the matrix $A$ defined
in (\ref{A}) now is:
$$A=\left(\begin{array}{cc}
a_2&a_{1}\\
0&a_2\\
0&0
\end{array}\right|
\left.\begin{array}{ccc}
 1&0+\mu_{1,4}a_2&0+\mu_{1,4}a_{1}+\mu_{1,5}a_2\\
a_{1}&1+\mu_{2,4}a_2&0+\mu_{2,4}a_{1}+\mu_{2,5}a_2\\
a_2&a_{1}+\mu_{3,4}a_2&1+\mu_{3,4}a_{1}+\mu_{3,5}a_2\end{array}\right|
$$
$$\left|\begin{array}{ccccc}
\mu_{1,4}+\mu_{1,5}a_{1}+\mu_{1,6}a_2&\cdots&
\mu_{1,n}+\mu_{1,n+1}a_1+\mu_{1,n+2}a_2&\mu_{1,n+1}+\mu_{1,n+2}a_1&\mu_{1,n+2}\\
\mu_{2,4}+\mu_{2,5}a_{1}+\mu_{2,6}a_2&\cdots&
\mu_{2,n}+\mu_{2,n+1}a_1+\mu_{2,n+2}a_2&\mu_{2,n+1}+\mu_{2,n+2}a_1&\mu_{2,n+2}\\
\mu_{3,4}+\mu_{3,5}a_{1}++\mu_{3,6}a_2&\cdots&
\mu_{3,n}+\mu_{3,n+1}a_1+\mu_{3,n+2}a_2&\mu_{3,n+1}+\mu_{3,n+2}a_1&\mu_{3,n+2}
\end{array} \right) .$$

With the usual notation, the affine parameterization of
$O^2_{\Lambda_0}(Z_y)$ yield that the matrix $B$ takes the form
$$B=\left(\begin{array}{cc}
a_2&a_1\\
\overline{a_2a_1}&\overline{a_1^2}-a_2\\
-\overline{a_2^2}&-2\overline{a_2a_1}
\end{array}\right|
\left.\begin{array}{ccc}
 1&0&0\\
0&1&0\\
0&0&1\end{array}\right|
\left.\begin{array}{ccc}
 *&\dots&*\\
*&\dots&*\\
*&\dots&*\end{array}\right)
$$

We also write the following  relevant parts of the
affine parameterization of $O^2_{\Lambda_0}(Z_y)$
\begin{enumerate}
\item\label{eq1} $p_{0,1,2}=p_{0,1,3}=0$,
\item\label{eq2} $p_{0,1,4}=\overline{a_2^2}=-p_{0,2,3}$,
\item\label{eq3} $p_{i,j,k}=0$ if $i,j,k\ne2,3,4$,
\item\label{eq4} $p_{0,3,i}=\overline{a_2\mu_{3,i-1}}=-p_{1,2,i}$
for
$i=5,\dots,n+2$ 
\end{enumerate}

Equalities \ref{eq1}. are precisely the vanishing of two of the
three minors of the left block of $B$. If it were
$\overline{a_2^2}\ne0$, also the third minor would be zero, i.e.
$p_{0,1,4}=0$. Thus equality \ref{eq2}. implies
$\overline{a_2^2}=0$. Hence we have $\overline{a_2^2}=0$ in any
case. Since
$p_{0,1,2}=0$, also $\overline{a_2a_1}=0$. Since also
$p_{0,1,3}=0$, either
$a_2$ or $\overline{a_1^2}-a_2$ are zero. With these vanishings in
mind, equations \ref{eq3}. say also that the submatrix of $B$ after
removing the central identity block has rank at most two. This
yields three possibilities for $B$. One of them corresponds
exactly to the set ${\mathcal C}$, while each the other two cases
splits, using equations \ref{eq4}., into two different
possibilities, which are inside the sets ${\mathcal A}$,
${\mathcal B}$ or ${\mathcal C}$.

\bigskip

{\bf CASE $r=1$:} The analogous of the matrix $A$ defined
in (\ref{A}) now is:
$$A=\left(\begin{array}{c}
a_{1}\\
0\\
0
\end{array}\right|
\left.\begin{array}{ccc}
 1&0&0+\mu_{1,4}a_{1}\\
a_{1}&1&0+\mu_{2,4}a_{1}\\
0&a_{1}&1+\mu_{3,4}a_{1}\end{array}\right|
\left.\begin{array}{ccc}
\mu_{1,4}+\mu_{1,5}a_{1}&\cdots &\mu_{1,n+2}\\
\mu_{2,4}+\mu_{2,5}a_{1}&\cdots &\mu_{2,n+2}\\
\mu_{3,4}+\mu_{3,5}a_{1}&\cdots &\mu_{3,n+2}
\end{array} \right) .$$
{}From this, we obtain our result as above.
\end{proof}

\begin{rem} \rm{The statement of Theorem \ref{O2Zramo} can be
improved. For example, when $r=1$ we know that equality holds, even
for arbitray $d$, although we preferred to write only the part we
need. }
\end{rem}

\begin{theorem}\label{osc 2 alla veronese} 
Let $\Lambda_{0}\in \GG (n-1,n+2)$ such that the intersection between
$\Lambda_{0}$ and $\Sigma$ in $\PP {n-1}$ is a zero-dimensional
scheme with support on $\{y_{1}, \ldots , y_{k}\}\subset\Sigma$
and degree $n$ such that each point $y_{i}$ has multiplicity
$r_{i}$ and
$\sum_{i=1}^{k}r_{i}=n$ (obviously $1\leq k \leq n$). Then, for
any $\Lambda \in O^{2}_{\Lambda_{0}}(V)\cap\GG (n-1,n+2)$, there
are two possibilities:
\begin{enumerate}
\item if $\dim(<\Lambda, \Lambda_{0}>)=n+1$ then there exist:
\begin{enumerate}
\item $y_{i_{1}}, y_{i_{2}}\in \Lambda_{0}\cap \Sigma$ such
that $\Lambda \cap \Sigma = \{ r_{1}y_{1}, \dots ,
(r_{i_{1}}-1)y_{i_{1}}, \dots , (r_{i_{1}}-1)y_{i_{2}}, \dots ,
r_{k}y_{k} \}$ and $\Lambda \cap \Lambda_{0}=<\Lambda \cap
\Sigma>$;
\item $Q'_{1}\in O_{y_{i_{1}}}^{r_{i_{1}}}(\Sigma),Q'_{2}\in
O_{y_{i_{2}}}^{r_{i_{2}}}(\Sigma)$ such that
$(r_{i_{1}}+1)y_{i_{1}}\in <\Lambda ,
Q'_{1}>,(r_{i_{2}}+1)y_{i_{2}}\in <\Lambda , Q'_{2}>$
\end{enumerate}
\item\label{2} if $\dim(<\Lambda, \Lambda_{0}>)=n$ then
\begin{enumerate}
\item either $\Lambda\in T_{\Lambda_0}V$;
\item\label{b} or there exist $y_{i_{1}}, y_{i_{2}}\in
\Lambda_{0}\cap \Sigma$ such that
$<r_1y_1,\dots,\widehat{r_{i_1}y_{i_1}},\dots,
\widehat{r_{i_2}y_{i_2}},\dots,r_ky_k>\subset\Lambda$
and
$<\Lambda ,\Lambda_{0}>=<\Lambda_{0},(r_{i_{1}}+2)y_{i_{1}}>\cap
<\Lambda_{0},(r_{i_{2}}+2)y_{i_{2}}>$.
\item\label{c} or there exists $y_{i}\in \Lambda_{0}\cap
\Sigma$ such that $<r_{1}y_{1},\dots ,\widehat{r_iy_{i}}, \dots
, r_{k}y_{k}>\subset \Lambda \subset
<\Lambda_{0},(r_{i}+2)y_{i}>$.
\end{enumerate}
\end{enumerate}

\end{theorem}

\begin{proof} For each $i=1,\dots,k$, let
${\mathcal A}_i,{\mathcal B}_i,{\mathcal C}_i\subset\GG(n-1,n+2)$
be the sets defined in the statement of Theorem \ref{O2Zramo} for
the point
$y_i\in\Sigma$. By Remark \ref{rami} and Theorem \ref{O2Zramo},
we have
\begin{equation}\label{sp osc Veronese}O^{2}_{\Lambda_{0}}(V)\cap
\GG (n-1,n+2)\subset\bigcap_{i=1}^{k} ({\mathcal A}_{i} \cup
{\mathcal B}_{i}\cup {\mathcal C}_{i}).\end{equation}

It is clear from \eqref{sp osc Veronese} that if $\Lambda \in
O^{2}_{\Lambda_{0}}(V)\cap\GG (n-1,n+2)$ the dimension of $<\Lambda,
\Lambda_{0}>$ is either $n$ or $n+1$. 

\begin{enumerate}
\item Assume that $\dim(<\Lambda, \Lambda_{0}>)=n+1$. We always
have $O^{r_{i}-1}_{y_{i}}(\Sigma)\subset <\Lambda ,
\Lambda_{0}>$. Moreover, if $\Lambda \in {\mathcal A}_{i}$, then
$<\Lambda,
\Lambda_{0}>=<\Lambda_{0}, O^{r_{i}+1}_{y_{i}}(\Sigma)>$ hence
$O^{r_{i}+1}_{y_{i}}(\Sigma)\subset <\Lambda , \Lambda_{0}>$.
Also, if $\Lambda \in {\mathcal B}_{i}$, then $<\Lambda,
\Lambda_{0}>=<\Lambda , \Lambda_{0}, O^{r_{i}}_{y_{i}}(\Sigma)>$
hence $O^{r_{i}}_{y_{i}}(\Sigma)\subset <\Lambda ,
\Lambda_{0}>$.  Since in $<\Lambda , \Lambda_{0}>$ there are at
most $n+2$ points of $\Sigma$ (counted with multiplicity), then
it follows that an intersection of $k$ sets of the form ${\mathcal
A}_{i}, {\mathcal B}_{j}, {\mathcal C}_{k}$ is larger that
$\{\Lambda_0\}$ only if is of the type
$ {\mathcal C}_{1} \cap \cdots \cap \widehat{{\mathcal
C}_{i_{1}}}\cap \cdots \cap
\widehat{{\mathcal C}_{i_{2}}}\cap \cdots \cap {\mathcal C}_{k}\cap
{\mathcal B}_{i_{1}}\cap {\mathcal B}_{i_{2}}$ or $ {\mathcal
C}_{1} \cap \cdots \cap
\widehat{{\mathcal C}_{i}}\cap \cdots \cap {\mathcal C}_{k}\cap
{\mathcal A}_{i}$. The latter is not possible because otherwise
$\Lambda\cap\Lambda_0$ would contain all the
$r_jy_j$ with $j\ne i$ and also a hyperplane of $<r_iy_i>$, and
hence its dimension would be at least $n-2$, which would imply
that $\dim(<\Lambda,
\Lambda_{0}>)<n+1$, contrary to our hypothesis.

 Assume for simplicity $i_{1}=1,i_{2}=2$. Now clearly
$<(r_{1}-1)y_{1}, (r_{2}-1)y_{2},r_{3}y_{3},\dots ,
r_{k}y_{k}>\subset\Lambda$ and there exist $Q'_{1}\in
O^{r_{1}}_{y_{1}}(\Sigma)$ and $Q'_{2}\in
O^{r_{2}}_{y_{2}}(\Sigma)$ such that $<(r_{1}+1)y_{1},
(r_{2}-1)y_{2},r_{3}y_{3},\dots , r_{k}y_{k}>\subset <\Lambda ,
Q'_{1}>$,
$<(r_{1}-1)y_{1}, (r_{2}+1)y_{2},r_{3}y_{3},\dots ,
r_{k}y_{k}>\subset <\Lambda , Q'_{2}>$. 
$\Lambda_{x}=<x,r_{1}y_{1}, (r_{2}-1)y_{2},r_{3}y_{3},\dots ,
r_{k}y_{k}>$, it follows from Corollary \ref{tau and grass}
that $\Lambda \in T_{\Lambda_{x}}V$. Hence $\Lambda$ should
belong to an infinite number of tangent space to $V$, and this
is absurd. Now it remains to show that $\Lambda \cap \Sigma$ is
not bigger than
$\{(r_{1}-1)y_{1}, (r_{2}-1)y_{2},r_{3}y_{3},\dots ,
r_{k}y_{k}\}$. Since
$\dim (\Lambda \cap \Lambda_{0})<n-2$ it cannot happen that
$r_{1}y_{1}$ or $r_{2}y_{2}$ belong to $\Lambda$. Then it is
sufficient to show that, for example, $(r_{3}+1)y_{3}\notin
\Lambda$ (if we allow $r_{3}=0$ then we are considering the
case $y_{3}\notin \Lambda_{0}$). Suppose for contradiction that
$(r_{1}-1)y_{1},(r_{2}-1)y_{2},(r_{3}+1)y_{3}, r_{4}y_{4},
\dots , r_{k}y_{k}\in \Lambda$. Hence  from Corollary
\ref{tau and grass} that $\Lambda \in T_{\Lambda_{1}}V$ where
$\Lambda_{1}=<r_{1}y_{1},
(r_{2}-1)y_{2},(r_{3}+1)y_{3},r_{4}y_{4}\dots , r_{k}y_{k}>$.
Analogously $\Lambda \in T_{\Lambda_{2}}V$ where
$\Lambda_{2}=<(r_{1}-1)y_{1},
r_{2}y_{2},(r_{3}+1)y_{3},r_{4}y_{4}\dots , r_{k}y_{k}>$. Since
$\Lambda$ corresponds to a degree three form, it is not
possible $\Lambda$ belongs to two different tangent spaces
because the elements of the tangent spaces corresponds to a
form containing a double factor.

\item Assume now that $\dim(<\Lambda, \Lambda_{0}>)=n$. Then the projection $\pi : \PP {n+2} \rightarrow \PP {2}$ from $\Lambda_{0}$ sends  $\Lambda$ in a point $P$ of $\PP 2$. Under this projection $\Sigma$ is sent to a conic $Q$ and the image $P_{i}$ of each
$y_{i}\in\Sigma$ is obtained by projecting $<(r_{i}+1)y_{i}>$. 
\\
If $\Lambda \in {\mathcal A}_{i}$ for some $i=1, \dots , k$, then
$\Lambda \subset <\Lambda_{0}, (r_{i}+2)y_{i}>$ and hence $P$
belongs to the tangent line in
$P_{i}$ to $Q$.
\\
If instead $\Lambda \in {\mathcal B}_{i}\backslash {\mathcal
C}_{i}$ for some
$i=1,
\dots , k$, then $\dim (\Lambda \cap <(r_{i}+1)y_{i}>)\geq
r_{i}-1$ and, since $\dim(\Lambda\cap\Lambda_0)\ge n-3$, then 
$<r_iy_{i}>$ is not contained in
$\Lambda$. Hence there exist $P'\in \Lambda \cap
<(r_{i}+1)y_{i}>\backslash <r_{i}y_{i}>$. Since $P'\in \Lambda$,
then $\pi(P')=P$, while since $P'\in 
<(r_{i}+1)y_{i}>\backslash <r_{i}y_{i}>$, also $\pi(P')=P_i$, so
that  $P=P_{i}$.
\\
{}From this description it is clear that intersections involving either three
${\mathcal A}_{i}$'s or one $({\mathcal B}_{j}\backslash {\mathcal
C}_{j})$'s and one
${\mathcal A}_{i}$'s or two ${\mathcal B}_{j}\backslash {\mathcal
C}_{j}$'s are empty. Let us study the remaining cases.
\begin{enumerate}

\item Assume first, after reordering, that $\Lambda \in {\mathcal
B}_{1}\cap {\mathcal C}_{2}\cap \cdots \cap {\mathcal C}_{k}$. By
definition
$<(r_{1}-1)y_{1}, r_{2}y_{2}, \dots ,r_{k}y_{k}>\subset \Lambda$
and there exists
$Q'\in <(r_{1}+1)y_{1}>$ such that $(r_{1}+1)y_{1}\in <Q',
\Lambda>$ hence $\Lambda \subset <Q',\Lambda>=<(r_{1}+1)y_{1},
r_{2}y_{2}, \dots ,r_{k}y_{k}>$. By Corollary \ref{tau and
grass},
$\Lambda \in T_{\Lambda_{0}}(V)$.

\item Assume now, after reordering,  $\Lambda \in {\mathcal
A}_{1}\cap {\mathcal A}_{2}\cap {\mathcal C}_{3}\cap\dots
\cap{\mathcal C}_{k}$. By definition
$<\Lambda,\Lambda_0>
\subseteq <\Lambda_{0},(r_{1}+2)y_{1}>\cap
<\Lambda_{0},(r_{2}+2)y_{2}>$ and this is an equality because
both the spaces on the left and the right hand side have the
same dimension $n$.

\item The last case $\Lambda \in  {\mathcal C}_{1}\cap \dots \cap
{\mathcal A}_{i}\cap
\dots \cap {\mathcal C}_{k}$ is trivial by definition.
\end{enumerate}
\end{enumerate}
\end{proof}

\section{Split Variety and the Grassmannian}\label{Split e Grass}

\begin{propos}\label{finalmente}
Let $\Lambda \in O_{\Lambda_{1}}^{2}(V) \cap
O_{\Lambda_{2}}^{2}(V)\cap \GG (n-1,n+2)$ for some
$\Lambda_{1},\Lambda_{2}\in V$, and assume $\dim (<\Lambda
,\Lambda_{1}>)=\dim (<\Lambda
,\Lambda_{2}>)=n+1$. Then there exist $s_1y_{1}, \ldots
, s_ky_{k}\in \Sigma$ with $\sum_{i=1}^{k}s_{i}=n-2$ such that
$\Lambda_{1}=<(s_{1}+1)y_{1}, (s_{2}+1)y_{2}, s_{3}y_{3},
s_{4}y_{4}\ldots , s_{k}y_{k}>$,  $\Lambda_{2}=<(s_{1}+1)y_{1},
s_{2}y_{2}, (s_{3}+1)y_{3}, \ldots , s_{k}y_{k}>$ and 
$$\Lambda=<(s_{1}+2)y_{1}, (s_{2}+2)y_{2}, s_{3}y_{3},
s_{4}y_{4}\ldots , s_{k}y_{k}>\cap $$
$$\cap<(s_{1}+2)y_{1}, s_{2}y_{2}, (s_{3}+2)y_{3}, \ldots ,
s_{k}y_{k}>\cap <s_{1}y_{1}, (s_{2}+2)y_{2}, (s_{3}+2)y_{3},
\ldots , s_{k}y_{k}>.$$
\end{propos}

\begin{proof} 
{}From Theorem \ref{osc 2 alla veronese} we can derive $\Lambda
\cap \Lambda_{1}=<\Lambda \cap \Sigma>=\Lambda \cap \Lambda_{2}$
and $\Lambda \cap \Sigma =\{ s_{1}y_{1},\dots,s_{k}y_{k} \}$
with $\sum_{i=1}^{k}s_{i}=n-2$ (the $s_{i}$'s do not have to be
necessarily different from zero). Moreover we know that
$\Lambda_{1}, \Lambda_{2}$ can be obtained form $<\Lambda \cap
\Sigma>$ increasing two $s_{i}$'s by $1$.

We show now that  the $s_{i}$'s we have to increase do not
correspond to four different $y_{i}$'s. Assume for
contradiction, up to reordering, that
$\Lambda_{1}=<(s_{1}+1)y_{1}, (s_{2}+1)y_{2}, s_{3}y_{3},
s_{4}y_{4}, s_{5}y_{5}, \dots , s_{k}y_{k}>$,
$\Lambda_{2}=<s_{1}y_{1}, s_{2}y_{2}, (s_{3}+1)y_{3},
(s_{4}+1)y_{4}, s_{5}y_{5}, \dots , s_{k}y_{k}>$. By Theorem
\ref{osc 2 alla veronese} there exist $Q'_{1}\in
O^{s_{1}+1}_{y_{1}}(\Sigma)$, $Q''_{1}\in
O_{y_{3}}^{s_{3}+1}(\Sigma)$ such that $(s_{1}+2)y_{1}\in
<\Lambda , Q'_{1}>$ and $(s_{3}+2)y_{3}\in <\Lambda , Q''_{1}>$,
hence the $(n+1)$-dimensional subspace $<\Lambda ,
Q'_{1},Q''_1>$ contains the  following $n+4$ points of $\Sigma$:
$(s_{1}+2)y_{1}, (s_{2}+1)y_{2}, (s_{3}+2)y_{3}, (s_{4}+1)y_{4},
s_{5}y_{5}, \dots , s_{k}y_{k}$, which is clearly a
contradiction. Hence we can assume, up to reordering, 
$$\Lambda_{1}=<(s_{1}+1)y_{1},(s_{2}+1)y_{2}, s_{3}y_{3},
s_{4}y_{4}, \dots , s_{k}y_{k}>$$ 
$$\Lambda_{2}=<(s_{1}+1)y_{1},s_{2}y_{2}, (s_{3}+1)y_{3},
s_{4}y_{4}, \dots , s_{k}y_{k}>.$$ 

By Theorem \ref{osc 2 alla veronese}, there exists
$Q'_1\in<(s_1+2)y_1>$ such that
$<(s_1+2)y_1>\subset<\Lambda,Q'_1>$. Since $\Lambda$ is a
hyperplane in $<\Lambda,Q'_1>$, we can find
$R'_1\in\Lambda\cap<(s_1+2)y_1>\setminus<s_1y_1>$. Analogously,
we can find $R'_2\in\Lambda\cap<(s_2+2)y_2>\setminus<s_2y_2>$
and $R'_3\in\Lambda\cap<(s_3+2)y_3>\setminus<s_3y_3>$. 

We claim that $<s_1y_1,\dots,s_ky_k,R'_1,R'_2>$ has dimension
$n-1$. Indeed, if $s_1y_1,\dots,s_ky_k,R'_1,R'_2$ were
dependent, the projection from $<s_1y_1,\dots,s_ky_k>$ would
produce a rational normal curve in $\PP4$ in which the tangent
lines at the image of $y_1$ and $y_2$ would meet at the image
of $R'_1$ (which would have the same image as $R'_2$), but this
is impossible. As a consequence of the claim,
$\Lambda=<s_1y_1,\dots,s_ky_k,R'_1,R'_2>$, so that it is
contained in $<(s_{1}+2)y_{1}, (s_{2}+2)y_{2}, s_{3}y_{3},
s_{4}y_{4},
 \dots , s_{k}y_{k}>$.

Analogously, $\Lambda\subset<(s_{1}+2)y_{1},
s_{2}y_{2}, (s_{3}+2)y_{3}, s_4y_{4}, \dots ,
s_{k}y_{k}>$ and $\Lambda\subset<s_{1}y_{1}, (s_{2}+2)y_{2},
(s_{3}+2)y_{3},
\ldots , s_{k}y_{k}>$. Therefore
$$\Lambda\subset<(s_{1}+2)y_{1}, (s_{2}+2)y_{2}, s_{3}y_{3},
s_{4}y_{4}\ldots , s_{k}y_{k}>\cap $$
$$\cap<(s_{1}+2)y_{1}, s_{2}y_{2}, (s_{3}+2)y_{3}, \ldots ,
s_{k}y_{k}>\cap <s_{1}y_{1}, (s_{2}+2)y_{2}, (s_{3}+2)y_{3},
\ldots , s_{k}y_{k}>.$$
We actually have an equality, since otherwise the usual
projection from $<s_1y_1,\dots,s_ky_k>$ would produce a rational
normal curve $\Sigma'\subset\PP4$, with points $y'_1,y'_2,y'_3$
such that the intersection $<T_{y'_1}\Sigma',T_{y'_2}\Sigma'>
\cap<T_{y'_1}\Sigma',T_{y'_3}\Sigma'>
\cap<T_{y'_2}\Sigma',T_{y'_3}\Sigma'>$ is more than a line. But
since $\Sigma'$ is homogeneous, the same would be true for any
choice of three points of $\Sigma'$, which is not true, as we
showed in Example \ref{altra componente Split3P2}.
\end{proof}

\begin{lemma}\label{quasi}
Let $\Lambda \in (O_{\Lambda_{1}}^{2}(V)\backslash T_{\Lambda_{1}}(V)) \cap (O_{\Lambda_{2}}^{2}(V)\backslash T_{\Lambda_{2}}(V))\cap
\GG (n-1,n+2)$ for some $\Lambda_{1},\Lambda_{2}\in V$, and assume $\dim
(<\Lambda ,\Lambda_{1}>)=\dim (<\Lambda , \Lambda_{2}>)=n$. Then
$\Lambda_1$ and $\Lambda_2$ have $n-1$ points of $\Sigma$ in
common (counted with multiplicity).
\end{lemma}

\begin{proof} We assume for contradiction that $\Lambda_1$ and
$\Lambda_2$ have at most $n-2$ points of $\Sigma$ in
common. Therefore $<\Lambda_1,\Lambda_2>$ contains at least $n+2$
points of $\Sigma$. This implies $\dim(<\Lambda_1,\Lambda_2>)\ge
n+1$. On the other hand, since $\dim (<\Lambda
,\Lambda_{1}>)=\dim (<\Lambda,\Lambda_{2}>)=n$, it follows that
$\dim(<\Lambda,\Lambda_1,\Lambda_2>)\le n+1$. As a consequence, 
$\dim(<\Lambda_1,\Lambda_2>)=n+1$, $\Lambda \subset <\Lambda_{1},
\Lambda_{2}>$ and $\Lambda_1$ and $\Lambda_2$ share exactly
$n-2$ points of $\Sigma$. 

We will write $\Lambda_1=<r_1y_1,\dots,r_ky_k>$, with
$r_1+\dots+r_k=n$. Since $\Lambda_1$ and $\Lambda_2$ share
$n-2$ points of
$\Sigma$, then $\Lambda_2$ is obtained by substracting two points to
$r_1y_1,\dots,r_ky_k$ and adding two more, maybe just
substracting or adding some multiplicities to the points. To
simplify the notation, we will include the points of
$\Lambda_2\setminus\Lambda_1$ in
$y_1,\dots,y_k$, so that maybe some $r_i$ (two at most) can be
zero.  From Theorem \ref{osc 2 alla veronese} we know that the
possible cases for $\Lambda_1$ and $\Lambda_2$ are those
described in
\ref{b}) and \ref{c}).

We exclude first the possibility that $\Lambda_1$ is in
case \ref{c}) of Theorem \ref{osc 2 alla veronese}. Otherwise,
up to reordering $r_2y_2,\dots,r_k,y_k\in\Lambda$ and
$\Lambda\subset<(r_1+2)y_1,r_2y_2,r_3y_3,\dots,y_k>$. Using
Proposition \ref{punti in osculanti}, we get
$r_2y_2,\dots,r_k,y_k\in\Lambda_2$. We have now two
possibilities (after probably reordering
$y_1,\dots,y_k$) for $\Lambda_2$, namely
$$<(r_1-2)y_1,(r_2+2)y_2,r_3y_3,r_4y_4,\dots,r_ky_k>$$ 
$$<(r_1-2)y_1,(r_2+1)y_2,(r_3+1)y_3,r_4y_4,\dots,r_ky_k>.$$
This gives the following respective possibilities for
$<\Lambda_1,\Lambda_2>$: 
$$<r_1y_1,(r_2+2)y_2,r_3y_3,r_4y_4,\dots,r_ky_k>$$
$$<r_1y_1,(r_2+1)y_2,(r_3+1)y_3,r_4y_4,\dots,r_ky_k>.$$
Observe that it cannot be $(r_2+1)y_2\in\Lambda$, since
Proposition \ref{punti in osculanti} would imply
$(r_2+1)y_2\in\Lambda_1$. Therefore, by part 2. of Theorem
\ref{osc 2 alla veronese} taking
$\Lambda_0=\Lambda_2$, we have 
$\Lambda\subset<\Lambda_2,(r_2+4)y_2>$ or
$\Lambda\subset<\Lambda_2,(r_2+3)y_2>$, depending on the two
possibilities for $\Lambda_2$. Having also in mind the
inclusion $\Lambda\subset<\Lambda_1,\Lambda_2>$, we get that
$\Lambda$ is contained in one of the following (corresponding
to the two possibilities for $\Lambda_2$):
$$<(r_1+2)y_1,r_2y_2,r_3y_3\dots,r_ky_k>\cap
<r_1y_1,(r_2+2)y_2,r_3y_3,\dots,r_ky_k>\cap
<(r_1-2)y_1,(r_2+4)y_2,r_3y_3,\dots,r_ky_k>$$
$$<(r_1+2)y_1,r_2y_2,r_3y_3,\dots,r_ky_k>\cap
<r_1y_1,(r_2+1)y_2,(r_3+1)y_3,r_4y_4,\dots,r_ky_k>\cap$$
$$\cap<(r_1-2)y_1,(r_2+3)y_2,(r_3+1)y_3,r_4y_4,\dots,r_ky_k>$$
which is a contradiction by Example \ref{non in Split3P2}
(since $\Lambda$ is in two different osculating spaces to $V$,
it necessarily belongs to 
$\Split_3(\PP n)$).

We are thus reduced to the possibility that $\Lambda_1$ is in case
\ref{b}) of Theorem \ref{osc 2 alla veronese}.
Therefore, up to reordering,
$r_3y_3,\dots,r_ky_k\in\Lambda$ and
$\Lambda\subset<(r_1+2)y_1,r_2y_2,r_3y_3,\dots,r_ky_k>\cap<r_1y_1,(r_2+2)y_2,r_3y_3,\dots,r_ky_k>$.
By Proposition \ref{punti in osculanti}, it follows that
$r_3y_3,\dots,r_ky_k\in\Lambda_2$. Hence there are four
possibilities (after probably reordering $y_1,\dots,y_k$) for
$\Lambda_2$, namely
$$<(r_1-2)y_1,(r_2+2)y_2,r_3y_3,r_4y_4,r_5y_5,\dots,r_ky_k>$$ 
$$<(r_1-2)y_1,(r_2+1)y_2,(r_3+1)y_3,r_4y_4,r_5y_5,\dots,r_ky_k>$$
$$<(r_1-1)y_1,(r_2-1)y_2,(r_3+2)y_3,r_4y_4,r_5y_5,\dots,r_ky_k>$$
$$<(r_1-1)y_1,(r_2-1)y_2,(r_3+1)y_3,(r_4+1)y_4,r_5y_5,\dots,r_ky_k>.$$
As before, Proposition \ref{punti in osculanti} implies that it
cannot be
$(r_2+1)y_2\in\Lambda$ or $(r_3+1)y_3\in\Lambda$. Hence, by part 2.
of Theorem \ref{osc 2 alla veronese} applied for
$\Lambda_0=\Lambda_2$ in the four possibilities above we have,
respectively,
$$\Lambda\subset<\Lambda_2,(r_2+4)y_2>=
<(r_1-2)y_1,(r_2+4)y_2,r_3y_3,r_4y_4,\dots,r_k>$$
$$\Lambda\subset<\Lambda_2,(r_2+3)y_2>=
<(r_1-2)y_1,(r_2+3)y_2,(r_3+1)y_3,r_4y_4,\dots,r_k>$$
$$\Lambda\subset<\Lambda_2,(r_3+4)y_3>=
<(r_1-1)y_1,(r_2-1)y_2,(r_3+4)y_3,r_4y_4,\dots,r_k>$$
$$\Lambda\subset<\Lambda_2,(r_3+3)y_3>=
<(r_1-1)y_1,(r_2-1)y_2,(r_3+3)y_3,(r_4+1)y_4,r_5y_5,\dots,r_k>.$$
Since we also have
$\Lambda\subset<(r_1+2)y_1,r_2y_2,r_3y_3,\dots,r_ky_k>
\cap<r_1y_1,(r_2+2)y_2,r_3y_3,\dots,r_ky_k>$,
we get a contradiction from Example \ref{non in Split3P2}.
\end{proof}

\begin{propos}\label{lemma osc} Let $\Lambda \in
(O_{\Lambda_{1}}^{2}(V)\backslash T_{\Lambda_{1}}(V)) \cap
(O_{\Lambda_{2}}^{2}(V)\backslash T_{\Lambda_{2}}(V))\cap
\GG (n-1,n+2)$ for some $\Lambda_{1},\Lambda_{2}\in V$, and assume $\dim
(<\Lambda ,\Lambda_{1}>)=\dim (<\Lambda , \Lambda_{2}>)=n$. If
$\Lambda_1$ and $\Lambda_2$ do have $n-1$ points of $\Sigma$ in
common, also $\Lambda$ contains those points.
\end{propos}

\begin{proof} 
Since, by hypothesis, the intersection of $\Lambda_1$
and $\Lambda_2$ has dimension $n-2$, and also the intersection
of $\Lambda$ with each of them has dimension $n-2$, it follows
that there are two possibilities:

--Either $\Lambda$ contains the intersection of
$\Lambda_1,\Lambda_2$, hence their $n-1$ common points of
$\Sigma$.

--Or $\Lambda$ is contained in the $n$-dimensional span of
$\Lambda_1,\Lambda_2$. By Theorem \ref{osc 2 alla veronese}, in
any case there exists $y_1\in\Sigma\cap\Lambda_1$ such that
$\Lambda\subset<\Lambda_1,(r_1+2)y_1>$, where $r_1$ is the
intersection multiplicity at $y_1$ of $\Sigma$ and $\Lambda_1$.
Hence
$\Lambda\subset<\Lambda_1,(r_1+2)y_1>\cap<\Lambda_1,\Lambda_2>$.
Since $\Lambda\ne\Lambda_1$, necessarily
$<\Lambda_1,(r_1+2)y_1>$ contains
$<\Lambda_1,\Lambda_2>$, in particular the point of
$\Lambda_1\cap\Sigma$ that is not in $\Lambda_2$. Since the
hyperplane
$<\Lambda_1,(r_1+2)y_1>\subset\PP{n+2}$ cannot $n+3$ different
point of $\Sigma$, it follows that $(r_1+1)y_1\in\Lambda_2$. We
cannot have another $y'_1\ne y_1$ in $\Sigma\cap\Lambda_1$ such
that $\Lambda\subset<\Lambda_1,(r'_1+2)y'_1>$, because
the same reasoning would show $(r'_1+1)y'_1\in\Lambda_2$, which
contradicts the fact that $\Lambda_1$ and $\Lambda_2$ share
$n-1$ points of $\Sigma$. Therefore $\Lambda_1$ is in case 2.(c)
of Theorem \ref{osc 2 alla veronese}. The same reasoning for
$\Lambda_2$ shows that there exists $y_2\in\Sigma\cap\Lambda_2$
with multiplicity $r_2$ and such that $(r_2+1)y_2\in\Lambda_1$.
Moreover, $\Lambda_2$ is also in case 2.(c)
of Theorem \ref{osc 2 alla veronese}. But then, using again the
part 2.(c) of Theorem \ref{osc 2 alla veronese}, we deduce that
$\Lambda$ should contain $(r_1+1)y_1,(r_2+1)y_2$ and the other
$n-r_1-r_2$ common points of $\Sigma$, which is a contradiction.
\end{proof}

\begin{theorem} \label{split3conGrass} The intersection
between $\Split_{3} (\PP n)$ and $\GG (n-1, n+2)$ is 
$$\Split_{3}(\PP n)\cap \GG (n-1,n+2)=X_{n+1}\cup X_{n+2}$$
where
$$X_{n+1}=\{<Z+2y_1+2y_2>\cap<Z+2y_1+2y_3>\cap<Z+2y_2+2y_3>\ |\
Z\subset\Sigma,\  {\rm length}(Z)=n-2,\ y_1,y_2,y_3\in\Sigma\}$$
$$X_{n+2}=\{\Lambda\subset\GG(n-1,n+3)\ |\ 
{\rm length}(\Lambda\cap\Sigma)\ge n-1\}.$$
\end{theorem}

\begin{proof} We have $X_{n+1}\subset\Split_{3} (\PP n)$ by
Proposition \ref{Xn+1 in Split} and
$X_{n+2}\subset\Split_{3} (\PP n)$ by Corollary \ref{un
contenimento}. Hence
$X_{n+1}\cup X_{n+2}\subset \Split_{3}(\PP n)\cap \GG
(n-1,n+2)$.

Reciprocally, let $\Lambda\in\Split_{3}(\PP n)\cap \GG
(n-1,n+2)$. By Remark \ref{Split-osc}, either
$\Lambda\in\tau(V)\cap \GG(n-1,n+2)$ or $\Lambda\in
O^{2}_{\Lambda_{1}}(V)\cap O^{2}_{\Lambda_{2}}(V)\cap
O^{2}_{\Lambda_{3}}(V)$ for different subspaces
$\Lambda_1,\Lambda_2,\Lambda_3\in\GG(n-1,n+2)$. In the first
case, by Corollary \ref{tau and grass}, $\Lambda$ contains at
least $n-1$ points of $\Sigma$, so that $\Lambda\in X_{n+2}$.
We will thus assume $\Lambda\not\in\tau(V)$ and $\Lambda\in
O^{2}_{\Lambda_{1}}(V)\cap O^{2}_{\Lambda_{2}}(V)\cap
O^{2}_{\Lambda_{3}}(V)$. Theorem \ref{osc 2 alla veronese}
implies that the span of $\Lambda$ with each $\Lambda_i$ has
dimension $n+1$ or $n$. Hence for at least two of the subspaces,
say $\Lambda_1,\Lambda_2$, the dimensions of
$<\Lambda,\Lambda_1>$ and $<\Lambda,\Lambda_2>$ are the same. We
study separately the different possibilities:

If $\dim(<\Lambda,\Lambda_1>)=\dim(<\Lambda,\Lambda_2>)=n+1$,
by Proposition \ref{finalmente}, we have $\Lambda\in X_{n+1}$.

If $\dim(<\Lambda,\Lambda_1>)=\dim(<\Lambda,\Lambda_2>)=n$, by
Lemma \ref{quasi} it follows that $\Lambda_1,\Lambda_2$ have
$n-1$ points of $\Sigma$ in common, so that we are done by
Proposition \ref{lemma osc}.

\end{proof}


\section{Appendix}\label{cancellabile}

In this appendix we want to explore the following problem: is it
possible to detect when the $s$-th secant variety to $\Split_{d}(\PP
n)$ fills up the whole ambient space by just detecting when its
intersection with $\GG (n-1,n+d-1)$ is the whole Grassmannian? 

To test the validity of this method, one could replace
$\Split_{d}(\PP n)$ with $\nu_{d}(\PP n)$, for which the
dimensions of all secant varieties are known (see \cite{AH}). We
will see that in fact, the method perfectly works for $d=2$ and
any secant variety, and give some partial answer for any $d$ and
the second secant variety.

\begin{propos}\label{Grass e sec} The intersection between the
Grassmannian $\GG (n-1,n+1)$ and the variety
$\Sec_{r-1}(\nu_{2}(\PP n))$ 
is the set of all $(n-1)$-spaces of $\PP {n+1}$
that are $(n-r+1)$-secant to the rational normal curve
$\Sigma\subset\PP{n+1}$.
\end{propos}

\begin{proof} Assume first that a subspace
$\Lambda\subset\PP{n+1}$ contains a subscheme $Z\subset\Sigma$ of
length $n-r+1$. By Lemma \ref{contained}, we can find linear forms
$N_0,\dots,N_{r-1}\in K[X_0,\dots,X_n]$ such that $\Lambda$, as an
element of $\PP{}(K[X_0,\dots,X_n]_2)$ lies in
$\PP{}(K[N_0,\dots,N_{r-1}]_2)$. But now the $r$-th secant
variety of $\nu_2(\PP{}(K[N_0,\dots,N_{r-1}]_1)$ is the whole
$\PP{}(K[N_0,\dots,N_{r-1}]_2)$. Thus necessarily $\Lambda$
belongs to $\Sec_{r-1}(\nu_{2}(\PP n))$.

We just sketch the proof of the other inclusion (although the case
$r=2$ is an immediate consequence of Corollary \ref{tau and
grass}). The main idea for the proof is that, since $d=2$, the
Pl\"ucker space of $\GG (n-1,n+1)$ can be identified with the
space of classes of skew-symmetric matrices of order $n+2$, while
the space of homogeneous polynomials of degree two in $n+1$
variables can be regarded as the space of symmetric matrices of
order $n+1$. In this language, one can write down explicitly the
identification of these two spaces. Specifically, to any
skew-symmetric matrix 
\begin{equation*}A=\left(
\begin{array}{ccccc}
 0& p_{0,1} & \cdots &  p_{0,n+1} \\
  -p_{0,1} & 0 & \cdots  & p_{1,n+1} \\
  \vdots &  & \ddots  &  \vdots \\
  -p_{0,n+1} &- p_{1,n+1} & \cdots & 0 \\
\end{array}
\right).\end{equation*}
the corresponding symmetric matrix is
\begin{equation*}Q=\left(
\begin{array}{ccccc}
  p_{0,1} & p_{0,2} & p_{0,3} & \cdots &  p_{0,n+1} \\
  p_{0,2} & p_{1,2}+p_{0,3} & p_{1,3}+p_{0,4} & \cdots  & p_{1,n+1} \\
  p_{0,3} & p_{1,3}+p_{0,4} & p_{2,3}+p_{1,4}+p_{0,5} & \cdots & p_{2,n+1}  \\
  \vdots & \vdots  &  & &  \vdots \\
  p_{0,n+1} & p_{1,n+1} & \cdots &  \cdots & p_{n,n+1} \\
\end{array}
\right).\end{equation*}
Take then $\Lambda\in\GG(n-1,n+1)$ represented by a rank-two
matrix $A$ as above. If it belongs to $\Sec_{r-1}(\nu_{2}(\PP
n))$, this means that the corresponding matrix $Q$ has rank at
most $r$. It is then possible to verify that this is equivalent to
the fact that the system
$$A\left(\begin{array}{c}
  t_0^{n+1} \\ t_0^{n}t_1 \\ \vdots  \\ t_1^{n+1} \\
\end{array}\right)=
\left(\begin{array}{c}
  0 \\ 0 \\ \vdots  \\ 0 \\
\end{array}\right)$$
admits at least $n-r+1$ solutions in 
$\PP 1$, counted with multiplicity.
It follows that $A$ describes an
$(n-1)$-space of $\PP {n+1}$ that is $(n-r+1)$-secant to 
$\Sigma$.
\end{proof}

\begin{corol}\label{int2} The intersection between
$\Sec_{s-1}(\Split_2(\PP n))$ and $\GG (n-1,n+1)$ is
set-theoretically the locus $\{ \Lambda \in \GG (n-1,n+1)
\; | \; \Lambda \hbox{ is } (n-2s+1)-\hbox{secant to }
\nu_{n+1}(\PP 1) \}$. 
\end{corol}

\begin{proof} This is a consequence of the previous proposition
and of the observation that, since $\Split_2(\PP n)=\{ Q\in
M_{n+1}(K)
\hbox{ s.t. $Q$ is symmetric and } \rk (Q)=2\}$ and the elements
of $\Split_{2}(\PP n)$  are of the form 
$[L_1\cdot L_2]$ with $L_1,L_2\in R_1$, then
$\Sec_{s-1}(\Split_2(\PP n))=\{ [L_1L_2+\cdots
+L_{2s-1}L_{2s}]\in \PP {}(R_2) \; | \; L_i\in R_1 \hbox{ for }
i=1,\ldots ,2s \}$ is the set of all symmetric matrices of
$M_{n+1}(K)$ of rank at most $2s$.
\end{proof}

\begin{rem}{\rm
Observe that, the previous results show that the technique
proposed at the beginning of this appendix works for $\nu_{2}(\PP
n)$ and $\Split_{2}(\PP n)$. Indeed, $\Sec_{r-1}(\nu_{2}(\PP
n))=\PP{n(n+3)\over2}$ if and only if $r\ge n+1$, which is
equivalent (by Proposition \ref{Grass e sec}) to
$\Sec_{r-1}(\nu_{2}(\PP n))\cap\GG(n-1,n+1)=\GG(n-1,n+1)$.
Similarly, $\Sec_{s-1}(\Split_{2}(\PP
n))=\PP{n(n+3)\over2}$ if and only if $s\ge\frac{n+1}{2}$ (because
$\Sec_{s-1}(\Split_{2}(\PP n))$ can be interpreted as the space of
symmetric matrices of rank at most $2s$) and this is equivalent
(by Corollary \ref{int2}) to $\Sec_{s-1}(\Split_{2}(\PP
n))\cap\GG(n-1,n+1)=\GG(n-1,n+1)$.
}\end{rem}

We end by presenting some generalizations of Proposition
\ref{Grass e sec}. We need some preliminary results.

\begin{lemma} \label{rette in Grass} Let  $\Lambda_{1},
\Lambda_{2}\in \nu_d(\PP n)$ such that the line spanned by them
is contained in $\GG (n-1,n+d-1)$. Then
$\Lambda_{1}$ and $\Lambda_{2}$ share at least $n-1$ points of
$\Sigma$.
\end{lemma}

\begin{proof} Since the line spanned by $\Lambda_{1},
\Lambda_{2}$
is contained in $\GG (n-1,n+d-1)$, they belong to a pencil of
subspaces. Hence the span of $\Lambda_{1},
\Lambda_{2}$ in $\PP{n+d-1}$ is a linear space of dimension $n$.
The hypothesis $\Lambda_{1}, \Lambda_{2}\in \nu_d(\PP n)$, implies
that $\Lambda_{1},
\Lambda_{2}$ contain each $n$ points of $\Sigma$. Since
$<\Lambda_{1},
\Lambda_{2}>$ can contain at most $n+1$ points of $\Sigma$, the
result follows readily.
\end{proof}

\begin{propos} \label{forme binarie} Let $N_0,N_1$ be two
linear forms of $K[x_{0}, \ldots ,x_{n}]$;
then $\GG (n-1,n+2) \cap\PP{}(K[N_0,N_1]_3)= \{ \Lambda \in
\GG(n-1, n+2)\ \ | \; \deg(\Lambda \cap \Sigma) \geq
n-1\}$.
\end{propos}

\begin{proof} 
Take $\Lambda\in\GG(n-1,n+2)$. If $\Lambda \cap \Sigma$ contains
a subscheme $Z\subset\Sigma$ of length $n-1$, Lemma
\ref{contained} implies that there exist linear forms
$N'_0,N'_1\in K[x_{0},\ldots ,x_{n}]$ such that
$\GG (n-1,n+2) \cap\PP{}(K[N'_0,N'_1]_3)= \{ \Lambda \in
\GG(n-1, n+2)\ \ | \; \Lambda \cap \Sigma\supset Z\}$. In
particular, $N_0,N_1\in K[N'_0,N'_1]$, so that
$K[N_0,N_1]=K[N'_0,N'_1]$ and one of the wanted inclusions
follows.

Reciprocally, assume $\Lambda\in\PP{}(K[N_0,N_1]_3)$. Then we can
consider the twisted cubic $C\subset\PP{}(K[N_0,N_1]_3)$ defined
by the classes of the type $(\alpha N_0
+\beta N_1)^3\in K[N_{0}, N_{1}]_{3}$. If $\Lambda\in C$, in
particular $\Lambda\in\nu_3(\PP n)$, so that it contains $n$
points of $\Sigma$. If $\Lambda\not\in C$, then it belongs to a
bisecant (or tangent) line to $\Sigma$. This line is thus
trisecant to $\GG(n-1,n+2)$, hence it is contained in
$\GG(n-1,n+2)$. The other inclusion follows now from Lemma
\ref{rette in Grass}.
\end{proof}

\begin{corol} If $M \in K[N_{0}, N_{1}]_{3} \cap \GG(n-1,n+2)$,
with $N_{0}, N_{1}$ generic linear forms, then $M\in
\nu_{3}(\PP n)$.
\end{corol}

\begin{proof} If $M$ is a binary form contained into the Grassmannian $\GG(n-1,n+2)$, then by Proposition \ref{forme binarie} the linear forms $N_{0}, N_{1}$ must be ``special'', i.e. they have at least $n-1$ roots in common.
\end{proof}

\begin{lemma}\label{lemma} Let $A,B\in \nu_{d}(\PP n)$. If
there exists a point $C \in \Sec_1(\nu_d(\PP n))\cap \GG
(n-1,n+d-1)$ such that $C\in <A,B>\smallsetminus \nu_{d}(\PP
n)$, then $<A,B>\subset \GG (n-1,n+d-1)$.
\end{lemma}

\begin{proof} The set of the three points  $\{A,B,C\}$ is contained in the
intersection $<A,B>\cap\, \GG (n-1,n+d-1)$. Since the Grassmannian is an
intersection of quadrics, it cannot exist a point $D\in<A,B>$ but $D\notin \GG (n-1,n+d-1)$ then $<A,B>\subset \GG (n-1,n+d-1)$.
\end{proof}

\begin{propos}\label{hint} The intersection between
$\Sec_1(\nu_d(\PP n))$ and $\GG(n-1,n+d-1)$ is contained in
$\{\Lambda \in \GG(n-1,n+d-1) \; | \; \deg (\Lambda \cap
\Sigma)\geq n-1\}$.
\end{propos}

\begin{proof}
Let us take a point $A\in \Sec_1(\nu_d(\PP n)\cap \GG
(n-1,n+d-1))\smallsetminus \nu_d(\PP n)$, then there exist $\pi_{1},\pi_{2}\in
\nu_{d}(\PP n)$ such that $A\in <\pi_{1},\pi_{2}>$. Since
$\nu_{d}(\PP n)$ is the locus of the $(n-1)$-spaces of $\PP
{n+d-1}$ that are $n$-secant to $\Sigma$, there exist
$P_{1},\ldots ,P_{n},Q_{1}, \ldots , Q_{n}\in \Sigma$
such that $\pi_{1}=<P_{1}, \ldots ,P_{n}>$ and $\pi_{2}=<Q_{1},
\ldots ,Q_{n}>$.
 Therefore $<\pi_{1},\pi_{2}>\subset
(\Sec_1(\nu_d(\PP n))\subset \Split_d(\PP n)$. By the Lemma
\ref{lemma} we have that $<\pi_{1},\pi_{2}>\subset \GG (n-1,n+d-1)$. The span $<\pi_{1},\pi_{2}>$ parameterizes a pencil of $(n-1)$-spaces contained in $\PP n\subset \PP {n+d-1}$ and  containing a
$\PP {n-2}$. Then $P_1,\ldots ,P_n,Q_1,\ldots ,Q_n$ lie on a
$\PP n$ instead of being generic in $<\Sigma>=\PP
{n+d-1}$, hence $\sharp\{ P_1,\ldots ,P_n,Q_1,\ldots ,Q_n
\}=n+1$. \end{proof}

\begin{propos}\label{grass e sec1 nu3}
Let $V=\nu_{3}(\PP n)\subset \GG (n-1,n+2)$, then  
$$\Sec_{1}(V) \cap \GG (n-1,n+2)= \{\Lambda \in \GG (n-1,n+2)\;
| \; \deg(\Sigma\cap \Lambda ) \geq n-1 \}.$$ 
\end{propos}

\begin{proof}
Proposition \ref{hint} presents one inclusion. Let's then prove
that $\{\Lambda \in \GG (n-1,n+2)\; | \; \deg(\Sigma\cap \Lambda
) \geq n-1 \}\subseteq \Sec_{1}(\nu_{3}(\PP n))\cap \GG
(n-1,n+2)$.

Let $\Lambda\in \GG (n-1,n+2)$ be a subspace containing a subscheme
$Z\subset\Sigma$ of length $n-1$.

Consider the projection $\pi: \PP {n+2}\rightarrow \PP 3$ from
$<Z>\subset \PP{n+2}$. Observe that all $\widetilde{\Lambda}\in
\GG (n-1, n+2)$ that intersect $\Sigma$ in degree $n$ are sent by
$\pi$ in the rational normal cubic $\Sigma'\subset\PP 3$, and
$\pi(\Lambda)=Q$ does not belong to such a cubic.
\\
A line $L\in \PP 3$ passing through $Q$ can be or tangent or
bisecant to the cubic.

If $L$ is the tangent line to $\Sigma'$ at a point $y'$, consider
$y\in\Sigma$  the point of $\Sigma$ whose image is $y$. Then
$<Z>\subset\Sigma\subset<Z+2y>$, so that
$\Lambda\in\tau(V)$.

If it is bisecant consider the $\PP {n}$ obtained as
$\pi^{-1}(L)=H\subset \PP {n+2}$. Since $L$ intersects the
rational normal cubic in two points, then $H$ contains two $\PP
{n-1}$'s, say $\Lambda_{1}$ and $\Lambda_{2}$, that intersect
$\Sigma$ in degree $n$, therefore from one side we can assume that
$H$ is spanned by them, from the other side $H$ can intersect
$\Sigma$ at most in degree $n+1$, hence $\Lambda_{1}$ and
$\Lambda_{2}$ have a $0$-dimensional scheme of degree $n-1$ on
$\Sigma$ in common.

Therefore we have found that an element $\Lambda\in \{\Lambda \in \GG (n-1,n+2)\; | \; \deg(\Sigma\cap \Lambda ) \geq n-1 \}$ belongs to a pencil of $\PP {n-1}$'s, that is a line in the Grassmannian and in particular such a line is spanned by two points belonging to $\GG (n-1,n+2)\cap V$, therefore $\Lambda \in \Sec_{1}(V)\cap \GG (n-1,n+2)$.
\end{proof}


\end{document}